\documentclass[a4paper,reqno]{amsart}

\title[On The Crepant Resolution Conjecture]
{On The Crepant Resolution Conjecture  In The Local Case}

\author{Tom Coates}
\address{Department of Mathematics\\
Imperial College London\\ 
180 Queen's Gate\\
London SW7 2AZ 
\\UK}
\email{t.coates@imperial.ac.uk}

\usepackage{amsrefs}
\usepackage{amsfonts}
\usepackage{amssymb}
\usepackage{amscd}
\usepackage{fullpage}
\usepackage[labelformat=empty]{subfig}
\usepackage{booktabs}
\usepackage{xypic}

\BibSpec{article}{%
+{}{\sc \PrintAuthors} {author}
+{,}{ \textrm} {title}
+{,}{ \textit} {journal}
+{,}{ } {volume}
+{}{ \IfEmptyBibField{journal}{}{\PrintDate{date}}} {transition}
+{,}{ } {pages}
+{,}{ } {status}
+{.}{} {transition}
+{}{ Preprint\IfEmptyBibField{date}{}{ \PrintDate{date}}, available at \tt} {eprint}
+{.}{} {transition}
}

\BibSpec{book}{%
+{}{\sc \PrintAuthors} {author}
+{,}{ \textit} {title}
+{.}{ } {series}
+{,}{ } {volume}
+{.}{ } {publisher}
+{,}{ } {place}
+{,}{ } {date}
+{.}{} {transition}
}

\BibSpec{collection.article}{
+{} {\sc \PrintAuthors} {author}
+{,} { \textrm } {title}
+{,} { in \it \PrintConference} {conference}
+{,} { pp.~} {pages}
+{.} {\PrintBook} {book}
+{,} { \PrintDateB} {date}
+{.} {} {transition}
}

\BibSpec{innerbook}{
+{.} { \emph} {title}
+{.} { } {part}
+{:} { \emph} {subtitle}
+{.} { } {series}
+{,} { } {volume}
+{.} { Edited by \PrintNameList} {editor}
+{.} { Translated by \PrintNameList}{translator}
+{.} { \PrintContributions} {contribution}
+{.} { } {publisher}
+{.} { } {organization}
+{,} { } {address}
+{,} { \PrintEdition} {edition}
+{,} { \PrintDateB} {date}
+{.} { } {note}
}

\allowdisplaybreaks[3]

\newcommand{\PP}{\mathbb{P}}

\newcommand{\CC}{\mathbb{C}}
\newcommand{\CCP}{\mathbb{CP}}
\newcommand{\ZZ}{\mathbb{Z}}
\newcommand{\RR}{\mathbb{R}}
\newcommand{\QQ}{\mathbb{Q}}

\newcommand{\Cstar}{{\CC^\times}}

\newcommand{\cF}{\mathcal{F}}

\newcommand{\cHX}{\mathcal{H}_\cX}
\newcommand{\cHY}{\mathcal{H}_Y}
\newcommand{\cHZ}{\mathcal{H}_\cZ}

\newcommand{\cO}{\mathcal{O}}
\newcommand{\ev}{\mathrm{ev}}

\newcommand{\be}{\mathbf{e}}

\newcommand{\lrl}{\mathbf{0}}

\renewcommand{\(}{\left(}
\renewcommand{\)}{\right)}

\DeclareMathOperator{\Res}{Res}
\newcommand{\fun}{\mathbf{1}}

\newcommand{\cB}{\mathcal{B}}
\newcommand{\cC}{\mathcal{C}}
\newcommand{\cE}{\mathcal{E}}
\newcommand{\cX}{\mathcal{X}}
\newcommand{\cY}{\mathcal{Y}}
\newcommand{\cZ}{\mathcal{Z}}
\newcommand{\cIX}{\mathcal{IX}}

\newcommand{\cLX}{\mathcal{L}_{\mathcal{X}}}
\newcommand{\cLY}{\mathcal{L}_Y}
\newcommand{\cLZ}{\mathcal{L}_\cZ}

\newcommand{\correlator}[1]{\left \langle #1 \right \rangle}
\newcommand{\smcorrelator}[1]{\big \langle #1 \big \rangle}

\newlength{\mybracketspacing}

\DeclareMathOperator{\Eff}{Eff}

\DeclareMathOperator{\NETT}{NETT}

\newcommand{\U}{\mathbb{U}}
\DeclareMathOperator{\QC}{QC}

\newcommand{\CR}{\underset{\scriptscriptstyle \text{CR}}{\cup}}

\newcommand{\JJ}{J^{\text{big}}}

\DeclareMathOperator{\fr}{frac}

\newcommand{\tw}{{\text{tw}}}

\theoremstyle{plain}

\newtheorem{thm}{Theorem}[section]

\newtheorem{proposition}[thm]{Proposition}
\newtheorem{conj}[thm]{Conjecture}

\newtheorem*{conj*}{Conjecture}
\newtheorem{cor}[thm]{Corollary}
\newtheorem*{cor*}{Corollary}

\theoremstyle{definition}

\newtheorem*{rem*}{Remark}

\newcommand{\cM}{\mathcal{M}}

\newcommand{\cI}{\mathcal{I}}
\renewcommand{\cH}{\mathcal{H}}
\renewcommand{\cL}{\mathcal{L}}

\newcommand{\GIT}[1]{/\!\!/_{\kern-.2em #1 \kern0.1em}}

\newcommand{\fp}{\mathfrak{p}}

\newcommand{\rank}{\operatorname{rank}}
\newcommand{\tti}{{\mathtt{i}}}

\begin{document}

\begin{abstract}
  In this paper we analyze four examples of birational transformations
  between local Calabi--Yau $3$-folds: two crepant resolutions, a
  crepant partial resolution, and a flop.  We study the effect of
  these transformations on genus-zero Gromov--Witten invariants,
  proving the Coates--Corti--Iritani--Tseng/Ruan form of the Crepant
  Resolution Conjecture in each case.  Our results suggest that this
  form of the Crepant Resolution Conjecture may also hold for more
  general crepant birational transformations.  They also suggest that
  Ruan's original Crepant Resolution Conjecture should be modified, by
  including appropriate ``quantum corrections'', and that there is no
  straightforward generalization of either Ruan's original Conjecture
  or the Cohomological Crepant Resolution Conjecture to the case of
  crepant partial resolutions. Our methods are based on mirror
  symmetry for toric orbifolds.
\end{abstract}

\maketitle

\section{Introduction}

Suppose that $\cX$ is an algebraic orbifold and that $\cY$ is an
orbifold or algebraic variety which is birational to $\cX$.  It is
natural to try to understand the relationship between the quantum
cohomology of $\cX$ and that of $\cY$.  In this paper we analyze four
examples of this situation --- two crepant resolutions, a crepant
partial resolution, and a flop --- which together exhibit some of the
range of phenomena which can occur.  The spaces that we consider are
local Calabi--Yau $3$-folds.  Our methods are based on mirror symmetry
for toric orbifolds.

The small quantum cohomology $\QC(\cX)$ of an orbifold $\cX$ is a
family of algebras depending on so-called \emph{quantum parameters}.
It arises in string theory as the chiral ring of the topological
$A$-model with target space $\cX$; from this point of view the quantum
parameters are co-ordinates on the stringy K\"ahler moduli space $\cM$
of $\cX$.  It is expected that the chiral rings form a family of
algebras over the whole of $\cM$ and that this family coincides with
$\QC(\cX)$ near the so-called \emph{large radius limit point}
$\lrl_\cX$ of $\cM$.  Other target spaces $\cY$ which are birational to
$\cX$ are expected to correspond to other limit points $\lrl_\cY$ of
$\cM$; this suggests that if $\cX$ and $\cY$ are birational then the
relationship between $\QC(\cX)$ and $\QC(\cY)$ should involve analytic
continuation in the quantum parameters.

A precise mathematical formulation of this was given by Ruan in his
influential Crepant Resolution Conjecture.  To state it, we need to
choose co-ordinates on (patches of) the stringy K\"ahler moduli space
$\cM$.  Suppose that $\cY \to X$ is a crepant resolution or partial
resolution of the coarse moduli space $X$ of $\cX$.  A choice of basis
$\varphi_1,\ldots,\varphi_r$ for $H^2(\cY;\CC)$ defines co-ordinates
$t_i$, $1 \leq i \leq r$ on $H^2(\cY;\CC)$, and hence defines
\emph{exponentiated flat co-ordinates} $q_i = e^{t_i}$, $1 \leq i \leq
r$, on a neighbourhood of $\lrl_\cY$ in $\cM$.  Similarly a choice of
basis $\phi_1,\ldots,\phi_s$ for $H^2(\cX;\CC)$ defines exponentiated
flat co-ordinates $u_i$, $1 \leq i \leq s$, near $\lrl_\cX$ in
$\cM$. We take $\varphi_1,\ldots,\varphi_r$ to be primitive integer
vectors on the rays of the K\"ahler cone for $\cY$ and
$\phi_1,\ldots,\phi_s$ to be primitive integer vectors on the rays of
the K\"ahler cone for $\cX$.  Because $\cY \to X$ is a (partial)
resolution there is a natural embedding $j:H^2(\cX;\QQ) \to
H^2(\cY;\QQ)$ which identifies the K\"ahler cone for $\cX$ with a face
of the K\"ahler cone for $\cY$.  We can therefore insist that
$j(\phi_i) = r_i \varphi_i$ for some rational numbers $r_i$, $1 \leq i
\leq s$.  The presence of these rational numbers reflects the fact
that the embedding $j$ does not in general identify the integer
lattices in $H^2(\cX;\QQ)$ and $H^2(\cY;\QQ)$.

The variables $q_1,\ldots,q_r$ and $u_1,\ldots,u_s$ thus defined are
the quantum parameters: the parameters on which the small quantum
cohomology algebras $\QC(\cY)$ and respectively $\QC(\cX)$ depend.
Ruan's Conjecture asserts that if $\cY \to X$ is a crepant resolution
then there are roots of unity $\omega_i$, $1 \leq i \leq r$, and a
choice of path of analytic continuation such that $\QC(\cX)$ is
isomorphic to the algebra obtained from $\QC(\cY)$ by analytic
continuation in the $q_i$ followed by the change of variables
\begin{equation}
  \label{eq:Ruancov}
  q_i =
  \begin{cases}
    \omega_i u_i^{r_i} & 1 \leq i \leq s \\
    \omega_i & s < i \leq r.
  \end{cases}
\end{equation}
One consequence of this is the Cohomological Crepant Resolution
Conjecture (CCRC) \cite{Ruan:CCRC}, which asserts that the Chen--Ruan
orbifold cohomology algebra of $\cX$ is isomorphic to the algebra
obtained from the small quantum cohomology of $\cY$ by analytic
continuation in the $q_i$ followed by the change of variables
\begin{equation*}
  q_i =
  \begin{cases}
    0 & 1 \leq i \leq s \\
    \omega_i & s < i \leq r.
  \end{cases}
\end{equation*}
These conjectures have been verified in a number of examples
\citelist{\cite{Perroni}\cite{Bryan--Graber--Pandharipande}\cite{Bryan--Graber}\cite{Boissiere--Mann--Perroni:1}\cite{CCIT:crepant1}\cite{Wise}\cite{CCIT:typeA}\cite{Bryan--Gholampour}\cite{Gillam}\cite{Boissiere--Mann--Perroni:2}}.

Recent progress in both mathematics and physics suggests, however,
that Ruan's Conjecture should be modified: that it is not an accurate
reflection of the physical picture.  In essence this is because when
we identify the family of algebras over $\cM$ (\emph{i.e.} the chiral
rings) with $\QC(\cX)$ and $\QC(\cY)$, we need to use exponentiated
flat co-ordinates near $\lrl_\cX$ and $\lrl_\cY$.  And even though the
family of algebras near $\lrl_\cX$ is related to the family of
algebras near $\lrl_\cY$ by analytic continuation, the analytic
continuation of exponentiated flat co-ordinates near $\lrl_\cX$
\emph{will not in general give exponentiated flat co-ordinates near
  $\lrl_\cY$}.  Thus we need also to analyze how the two co-ordinate
systems are related.  In some examples this has been done by
Aganagic--Bouchard--Klemm \cite{ABK} using \emph{ad hoc} methods and
by Coates--Corti--Iritani--Tseng \cite{CCIT:crepant1} in a more
systematic fashion; all of their examples satisfy the original Ruan
Conjecture.  One contribution of this paper is to give the first
example (Example II below) of a crepant resolution where the change in
exponentiated flat co-ordinate systems is sufficiently drastic that
the original Ruan Conjecture probably fails.  Our other examples
suggest that this failure cannot easily be fixed, and that a different
approach is needed.

In recent joint work with Corti, Iritani, and Tseng
\cite{CCIT:crepant1} we proposed\footnote{Similar ideas occurred in
  unpublished work of Ruan; an expository account can be found in
  \cite{Coates--Ruan}.} such a different approach, giving a new
conjectural picture of the relationship between the Gromov--Witten
theory of $\cX$ and that of $\cY$.  Our conjecture was phrased in
terms of Givental's symplectic formalism
\citelist{\cite{Coates--Givental:QRRLS}\cite{Givental:symplectic}}.
Genus-zero Gromov--Witten invariants of $\cX$ (and respectively $\cY$)
are encoded in a Lagrangian submanifold-germ $\cLX$ in a symplectic
vector space $\cHX$ (respectively $\cL_\cY \subset \cH_\cY$).  As
$\cLX$ and $\cL_\cY$ are germs of submanifolds it makes sense to
analytically continue them, and we conjectured the existence of a
linear symplectic isomorphism $\U:\cHX \to \cH_\cY$ satisfying some
quite restrictive conditions such that after analytic continuation we
have $\U(\cLX) = \cL_\cY$.  We also proved our conjecture when $\cX$
is one of the weighted projective spaces $\PP(1,1,2)$ or
$\PP(1,1,1,3)$ and $\cY \to X$ is a crepant resolution.

Our conjecture has consequences for quantum cohomology: it implies
both the Cohomological Crepant Resolution Conjecture and a modified
version of Ruan's Conjecture, each with the caveat that we must allow
the quantities $\omega_i$ to be arbitrary constants rather than roots
of unity.  (In the examples below the $\omega_i$ turn out to be roots
of unity and so the caveat disappears; Iritani has suggested an
attractive conceptual reason for this to be true in general
\cite{Iritani:integral}.)\phantom{.} The modified version of Ruan's
Conjecture has an additional hypothesis, that $\cX$ be
\emph{semi-positive}, and replaces the change of variables
\eqref{eq:Ruancov} by $ q_i = f_i(u_1,\ldots,u_s)$ where
\begin{equation*}
  f_i(u_1,\ldots,u_r) = 
  \begin{cases}
    \omega_i u_i^{r_i} + \text{higher order terms in $u_1,\ldots,u_r$}
    & 1 \leq i \leq s \\
    \omega_i + \text{higher order terms in $u_1,\ldots,u_r$} & s < i \leq r.
  \end{cases}
\end{equation*}
Thus we get a ``quantum corrected'' version of Ruan's original
conjecture.

In this paper we consider four examples:
\begin{enumerate}
\item[(I)] the crepant resolution of $\cX = \big[\CC^3/\ZZ_3\big]$, where $\ZZ_3$
  acts with weights $(1,1,1)$;
\item[(II)] the crepant resolution of the canonical bundle $\cX =
  K_{\PP(1,1,3)}$;
\item[(III)] the crepant partial resolution of $\cX = \big[\CC^3/\ZZ_5\big]$, where
  $\ZZ_5$ acts with weights $(1,1,3)$;
\item[(IV)] a toric flop with $\cX  = \cO_{\PP(1,2)}(-1)^{\oplus 3}$
  and $\cY = \cO_{\PP^2}(-1) \oplus \cO_{\PP^2}(-2)$.
\end{enumerate}
We prove the Coates--Corti--Iritani--Tseng/Ruan Crepant Resolution
Conjecture in each case.  This has implications as follows: \medskip
\begin{center}
  \begin{tabular}{@{}ccccc@{}} \toprule 
    & \multicolumn{4}{c}{Conjecture} \\ \cmidrule(r){2-5} 
    Example & CCIT/Ruan & CCRC & original Ruan &
    modified Ruan \\ \midrule  
    I & \checkmark& \checkmark& \checkmark & \checkmark\\
    II & \checkmark& \checkmark&?& \checkmark\\
    III & \checkmark&?&?&?\\
    IV & \checkmark&n/a&n/a&n/a\\ \bottomrule 
  \end{tabular} 
\end{center} \medskip I expect that wherever there is a ``?'' in this
table, the corresponding conjecture fails to hold, so that for example
the original form of Ruan's Conjecture fails in Example~II and the
modified form of Ruan's Conjecture fails in Example~III.  It is
difficult to prove these assertions, as this would involve ruling out
every possible choice of path of analytic continuation and all choices
of roots of unity, but I know of no reason to expect these conjectures
to hold.

In forthcoming work, Iritani will prove our form of the Crepant
Resolution Conjecture for all crepant birational transformations
between toric Deligne--Mumford stacks.  His method uses the full force
of the mirror Landau--Ginzburg model, the variation of semi-infinite
Hodge structure
\citelist{\cite{Barannikov:periods}\cite{Iritani:integral}\cite{Iritani:inprogress}}
associated to it, and the mirror theorem for toric Deligne--Mumford
stacks \cite{CCIT:stacks}.  Since all of our examples are included in
his discussion, it is natural to ask: ``what is the point of this
paper?''  The discussion here has quite modest goals, and is meant to
illustrate four points.  Firstly, these questions are \emph{not
  difficult}.  If $\cX$ is a toric orbifold $\cX$ and $\cY \to X$ is a
crepant resolution then the relationship between the quantum
cohomology of $\cX$ and that of $\cY$ can be determined
systematically, using well-understood methods from toric mirror
symmetry.  Secondly, our form of the Crepant Resolution Conjecture may
also hold, without significant change, for \emph{more general crepant
  birational transformations}: we see this here for two crepant
partial resolutions and a flop.  Thirdly, the method of proof
described here also \emph{applies without change} to the more general
crepant toric situation.  Finally, it seems likely that \emph{no
  na\"\i ve modification} of Ruan's original conjecture holds true; we
discuss this further in the next paragraph.  Along the way, we will
see two things which were perhaps already obvious: that Givental-style
mirror theorems are well-adapted to the analysis of toric birational
transformations, and that the methods of \cite{CCIT:crepant1} are
applicable to the (local) Calabi--Yau examples which are of greatest
interest to physicists \cite{ABK}.

The original conjecture of Ruan has an attractive simplicity, and one
might therefore ask whether our formulation of the Crepant Resolution
Conjecture is unneccessarily complicated and whether some simpler
statement holds \cite{Cadman--Cavalieri}.  The examples below
constitute some evidence that the answer to these questions is ``no''.
In Example~II below we see that quantum corrections to Ruan's original
conjecture are probably necessary, and in Example~III the situation is
even worse: there is probably not even a generalization of the
Cohomological Crepant Resolution Conjecture to partial resolutions
which involves only small (rather than big) quantum cohomology.  This
is related to the absence of a Divisor Equation for degree-two classes
from the twisted sectors, and is discussed further in
Section~\ref{sec:C3Z5}.

\subsection*{A Note on the Bryan--Graber Conjecture}

Jim Bryan and Tom Graber have recently given a generalization of
Ruan's Crepant Resolution Conjecture which applies to big quantum
cohomology, rather than just small quantum cohomology, under the
assumption that the orbifold $\cX$ involved satisfies a Hard Lefschetz
condition on Chen--Ruan orbifold cohomology.  We will not consider
this here, as none of our examples satisfy the Hard Lefschetz
condition.  But as the conclusion of the Bryan--Graber Conjecture
implies the Ruan Conjecture, Example II can be thought of as further
evidence that the Bryan--Graber Conjecture probably fails to hold
without the Hard Lefschetz assumption: see \cite{CCIT:crepant1} for
more on this.

\subsection*{Conventions} We will assume that the reader is familiar
with the Gromov--Witten theory of orbifolds.  This theory is
constructed in
\citelist{\cite{Chen--Ruan:orbifolds}\cite{Chen--Ruan:GW}\cite{AGV:1}\cite{AGV:2}};
a rapid overview can be found in \cite{CCLT}*{Section~2}.  We work in
the algebraic category, so for us ``orbifold'' means ``smooth
algebraic Deligne--Mumford stack over $\CC$''.  All of our examples
are non-compact, but they carry the action of a torus $T=\Cstar$ such
that the $T$-fixed locus is compact.  We therefore work throughout
with $T$-equivariant Gromov--Witten invariants, which in this setting
behave much as the Gromov--Witten invariants of compact orbifolds (see
\emph{e.g.}  \cite{Bryan--Graber}), and with $T$-equivariant
Chen--Ruan orbifold cohomology.  We always take the product of
$T$-equivariant Chen--Ruan classes using the Chen--Ruan product; when
we want to emphasize this, we will write the product as $\CR$.  The
degree of a Chen--Ruan class always means its orbifold or age-shifted
degree.

An expository account of our Crepant Resolution Conjecture and its
consequences can be found in \cite{Coates--Ruan}.  The reader should
take care when comparing the discussion in this paper with those in
\citelist{\cite{Bryan--Graber}\cite{Coates--Ruan}}, as here we measure
the degrees of orbifold curves using a basis of degree-two cohomology
classes chosen as above, whereas there the authors use a so-called
\emph{positive basis} for $H_2$.  Our choice of degree conventions
fits well with toric geometry, and this will be important below, but
we pay a price for our choice: the presence of the rational numbers
$r_i$ described above.

\subsection*{Outline of the Paper} In Section~\ref{sec:statement} we
fix notation and give a precise description of the conjecture which we
will prove.  In Section~\ref{sec:theory} we collect various
preparatory results, as well as describing how our conjecture implies
versions of Ruan's Conjecture and the Cohomological Crepant Resolution
Conjecture.  Examples~I--IV are in
Sections~\ref{sec:C3Z3}--\ref{sec:toricflop} respectively.

\subsection*{Acknowledgements} The author thanks Hiroshi Iritani for
many extremely useful conversations, and Alessio Corti, Yongbin Ruan,
and Hsian-Hua Tseng for a productive collaboration.  He thanks Andrea
Brini and Alessandro Tanzini for enlightening discussions of the
example $\cX = \big[\CC^3/\ZZ_4\big]$, and the referee for comments
which significantly improved the paper.

\section{Statement of the Conjecture}
\label{sec:statement}
In this section we give a precise statement of the conjecture that we will
prove.  Before we do so, we describe our general setup and fix
notation.

\subsection*{General Setup} 

Let $\cX$ be a Gorenstein orbifold with projective coarse moduli space
$X$ and let $\pi:Y \to X$ be a crepant resolution.  Assume that $\cX$,
$X$, and $Y$ carry actions of a torus $T = \Cstar$ such that both
$\pi$ and the structure map $\cX \to X$ are $T$-equivariant and such
that the $T$-fixed loci on $\cX$ and $Y$ are compact.  Let
$\CC[\lambda]$ denote the $T$-equivariant cohomology of a point, where
$\lambda$ is the first Chern class of the line bundle $\cO(1) \to
\CCP^\infty$, and let $\CC(\lambda)$ be its field of fractions.  Write
$H(\cX):= H^\bullet_{\text{CR},T}(\cX;\CC) \otimes \CC(\lambda)$ for
the localized $T$-equivariant Chen--Ruan orbifold cohomology of $\cX$,
and $H(Y) := H^\bullet_T(Y;\CC)\otimes \CC(\lambda)$ for the localized
$T$-equivariant cohomology of $Y$.  We work throughout over the field
$\CC(\lambda)$. The $\CC(\lambda)$-vector spaces $H(\cX)$ and $H(Y)$
carry non-degenerate inner products, given by
\begin{align*}
  (\alpha,\beta)_\cX := 
  \int_{\cIX^T} {i^\star (\alpha \cup I^\star \beta) \over
    \be(N_{\cIX^T/\cIX})}
  && \text{and} &&
  (\alpha,\beta)_Y := 
  \int_{Y^T} {j^\star (\alpha \cup \beta) \over
    \be(N_{Y^T/Y})}
\end{align*}
where $I$ is the canonical involution on the inertia stack $\cIX$ of
$\cX$; $i:\cIX^T \to \cIX$ and $j:Y^T \to Y$ are the inclusions of the
$T$-fixed loci in $\cIX$ and $Y$ respectively; $N_{\cIX^T/\cIX}$ and
$N_{Y^T/Y}$ are the normal bundles to the $T$-fixed loci; and $\be$ is the
$T$-equivariant Euler class.  Note that the $T$-equivariant Euler
classes are invertible over $\CC(\lambda)$.

\subsection*{The Symplectic Vector Space}

In what follows write $\cZ$ for either $\cX$ or $Y$, and write $Z$ for
the coarse moduli space of $\cZ$ (\emph{i.e.} for either $X$ or $Y$).
Introduce the symplectic vector space
\begin{align*}
  &\cHZ := H(\cZ) \otimes \CC(\!(z^{-1})\!) & \text{the vector space} \\
  &\Omega_\cZ(f,g) := \Res_{z=0} \big(f(-z),g(z)\big)_\cZ \, dz &
  \text{the symplectic form}
\end{align*}
and set $\cHZ^+ := H(\cZ) \otimes \CC[z]$, $\cHZ^- := z^{-1} H(\cZ)
\otimes \CC[\![z^{-1}]\!]$.  The polarization $\cHZ = \cHZ^+ \oplus
\cHZ^-$ identifies $\cHZ$ with the cotangent bundle $T^\star \cHZ^+$.
We regard $\cHZ$ as a graded vector space where $\deg z = 2$.

\subsection*{Degrees and Novikov Variables}

Fix a basis $\omega_1,\ldots,\omega_s$ for $H^2(\cX;\QQ)$ consisting
of primitive integer vectors on the rays of the K\"ahler cone for
$\cX$, and a basis $\omega'_1,\ldots,\omega'_r$ for $H^2(Y;\QQ)$
consisting of primitive integer vectors on the rays of the K\"ahler
cone for $Y$.  Note that $H^2(\cX;\QQ)$ is canonically isomorphic to
$H^2(X;\QQ)$, so we can regard $\omega_1,\ldots,\omega_s$ as
cohomology classes on $X$, and in our situation we can always insist
that $\pi^\star \omega_i = r_i \omega'_i$, $1 \leq i \leq s$, for some
rational numbers $r_i$.  We measure the degrees of orbifold curves
using the bases $\omega_i$ and $\omega_i'$.  Recall that a stable map
$f:\cC \to \cZ$ from an orbifold curve to $\cZ$ has a well-defined
degree in the free part
\[ 
H_2(Z;\ZZ)_{\text{free}}:= H_2(Z;\ZZ)/H_2(Z;\ZZ)_{\text{tors}}
\]
of $H_2(Z;\ZZ)$; we write $\Eff(\cZ) \subset H_2(Z;\ZZ)_{\text{free}}$
for the set of degrees of stable maps from orbifold curves to $\cZ$.  Given
an element $d \in \Eff(\cZ)$, set $d_i = \langle d, \omega_i \rangle$
if $\cZ = \cX$ and $d_i = \langle d, \omega'_i \rangle$ if $\cZ = Y$.
Note that the $d_i$ here are in general rational numbers.  Define $Q^d
:= Q_1^{d_1} \cdots Q_s^{d_s}$ where $d \in \Eff(\cX)$ and $Q^{d'} :=
Q_1^{d_1'} \cdots Q_r^{d_r'}$ where $d' \in \Eff(Y)$.  Here
$Q_1,Q_2,\ldots$ are formal variables called Novikov variables; the
number of Novikov variables associated with $\cZ$ is $b_2(Z)$, the
second Betti number of $Z$.

\subsection*{Bases and Darboux Co-ordinates}

We fix $\CC(\lambda)$-bases $\phi_0,\ldots,\phi_N$ and
$\phi^0,\ldots,\phi^N$ for $H(\cX)$ such that
\begin{itemize}
\item[(a)] $\phi_0$ is the identity element $\fun_\cX \in 
H(\cX)$;
\item[(b)] $\phi_1,\phi_2,\ldots,\phi_s$ are lifts to
  $T$-equivariant cohomology of
  $\omega_1,\omega_2,\ldots,\omega_s$;
\item[(c)] $(\phi_i,\phi^j)_\cX ={\delta_i}^j$;
\end{itemize}
and $\CC(\lambda)$-bases $\varphi_0,\ldots,\varphi_N$ and
$\varphi^0,\ldots,\varphi^N$ for $H(Y)$ such that
\begin{itemize}
\item[(d)] $\varphi_0$ is the identity element $\fun_Y \in H(Y)$;
\item[(e)] $\varphi_1,\varphi_2,\ldots,\varphi_r$ are lifts to
  $T$-equivariant cohomology of
  $\omega'_1,\omega'_2,\ldots,\omega'_r$;
\item[(f)] $(\varphi_i,\varphi^j)_Y ={\delta_i}^j$.
\end{itemize}
Conditions (b) and (e) here will be useful below when we discuss the
Divisor Equation.  Write
\begin{align*}
  \Phi_i = 
  \begin{cases}
    \phi_i & \text{if $\cZ = \cX$} \\
    \varphi_i & \text{if $\cZ = Y$} 
  \end{cases}
  && \text{and} &&
  \Phi^i = 
  \begin{cases}
    \phi^i & \text{if $\cZ = \cX$} \\
    \varphi^i & \text{if $\cZ = Y$.} 
  \end{cases}
\end{align*}
Then
\begin{equation}
  \label{eq:Darboux}
  \sum_{k \geq 0} q^\alpha_k \Phi_\alpha z^k + \sum_{l \geq 0}
  p_{\beta,l} \Phi^\beta (-z)^{-1-l}
\end{equation}
gives a Darboux co-ordinate system $\{q_{\alpha,k},p_{\beta,l} \}$ on
$\cHZ$; here and henceforth we use the summation convention on Greek
indices, summing repeated Greek (but not Roman) indices over the range
$0,1,\ldots,N$.

\subsection*{Gromov--Witten Invariants}

We use correlator notation for $T$-equivariant Gromov--Witten
invariants of $\cZ$, writing
\begin{equation}
  \label{eq:GW}
  \correlator{\alpha_1 \psi^{i_1},\ldots,\alpha_n \psi^{i_n}}^\cZ_{0,n,d} = 
  \int_{[\cZ_{0,n,d}]^{\text{vir}}} \prod_{k=1}^n
  \ev_k^\star(\alpha_k) \cdot \psi_k^{i_k}
\end{equation}
where $\alpha_1,\ldots,\alpha_n$ are elements of $H(\cZ)$ and
$i_1,\ldots,i_n$ are non-negative integers.  The cohomology classes
$\psi_1,\ldots,\psi_n$ here are the first Chern classes of the
universal cotangent line bundles on the moduli space $\cZ_{0,n,d}$ of
genus-zero $n$-pointed stable maps to $\cZ$ of degree $d \in
\Eff(\cZ)$.  The integral denotes the cap product with the
$T$-equivariant virtual fundamental class of $\cZ_{0,n,d}$: we discuss
this further in the next paragraph.  The right-hand side of equation
\eqref{eq:GW} is defined in \S8.3 of \cite{AGV:2} where it is denoted
$\correlator{\tau_{i_1} (\alpha_1),\ldots,\tau_{i_n}
  (\alpha_n)}_{0,d}$; our choice of notation allows compact
expressions for many important quantities, such as
\begin{align*}
  \correlator{{\alpha \over
      z-\psi}}^\cZ_{0,1,d} && \text{for} &&
  \sum_{m \geq 0} {1 \over z^{m+1}}
  \correlator{\alpha \psi^m \vphantom{\big\vert}}^\cZ_{0,1,d},
\end{align*}
as correlators are multilinear in their entries.

\subsection*{Twisted Gromov--Witten Invariants}

In most of the examples we consider below, $\cZ$ will be the total
space of a concave vector bundle $\cE$ over a compact orbifold (or
manifold) $\cB$, and the $T$-action on $\cZ$ will rotate the fibers of
$\cE$ and cover the trivial action on $\cB$.  That $\cE$ is concave
means that $H^0(\cC,f^\star \cE) = 0$ for all stable maps $f:\cC \to
\cB$ of non-zero degree.  This implies that stable maps to $\cE$ of
non-zero degree all land in the zero section and so, for $d \ne 0$,
the moduli space $\cZ_{0,n,d}$ coincides as a scheme with
$\cB_{0,n,d}$.  The natural obstruction theories on $\cZ_{0,n,d}$ and
$\cB_{0,n,d}$ differ, though, and the $T$-equivariant virtual
fundamental classes satisfy
\[
[\cZ_{0,n,d}]^{\text{vir}} = 
[\cB_{0,n,d}]^{\text{vir}} \cap \be(\text{Obs}_{0,n,d})
\]
where $\be$ is the $T$-equivariant Euler class and
$\text{Obs}_{0,n,d}$ is the vector bundle over $\cB_{0,n,d}$ with
fiber at a stable map $f:\cC \to \cB$ equal to $H^1(\cC,f^\star \cE)$.
Thus
\[
\int_{[\cZ_{0,n,d}]^{\text{vir}}} ( \cdots ) = 
\int_{[\cB_{0,n,d}]^{\text{vir}}} ( \cdots ) \cup \be(\text{Obs}_{0,n,d}).
\]
This means that Gromov--Witten invariants of $\cZ$ coincide with
\emph{twisted Gromov--Witten invariants} \cite{Coates--Givental:QRRLS,
  CCIT:computing} of $\cB$ where the twisting characteristic class is
the inverse $T$-equivariant Euler class $\be^{-1}$ and the twisting
bundle is $\cE$: this is explained in detail in \cite{CCIT:computing}.
Results of \cite{CCIT:computing} allow us to compute these twisted
Gromov--Witten invariants in terms of the ordinary Gromov--Witten
invariants of $\cB$, a fact which we exploit repeatedly below.

In the exceptional case $d=0$, the moduli space $\cZ_{0,n,d}$ is
non-compact and so we need to say what we mean by the integral in
\eqref{eq:GW}.  Since $\cZ_{0,n,d}$ carries a $T$-action with compact
fixed set, we can define the integral using the virtual localization
formula of Graber--Pandharipande \cite{Graber--Pandharipande}; note
that we could do this in the case $d \ne 0$, too, and this would
reproduce the definition which we just gave.

\subsection*{Gromov--Witten Potentials}

The genus-zero Gromov--Witten potential $F^0_\cZ$ is a generating
function for certain genus-zero Gromov--Witten invariants of $\cZ$.
It is a formal power series in variables $\tau^a$, $0 \leq a \leq N$,
and the Novikov variables $Q_i$, $1 \leq i \leq b_2(Z)$, defined by
\begin{equation}
  \label{eq:GWpotentialwithQ}
  F^0_\cZ = \sum_{n \geq 0} 
  \sum_{d \in \Eff(\cZ)} {Q^d \over n!}
  \smcorrelator{\overbrace{\tau,\tau,\ldots,\tau}^{\text{$n$
        times}}}^\cZ_{0,n,d}, 
\end{equation}
where $\tau = \tau^\alpha \Phi_\alpha$.  Since correlators are
multilinear, the expression
$\correlator{\tau,\tau,\ldots,\tau}^\cZ_{0,n,d}$ expands into a
polynomial in the variables $\tau^a$.  The second summation here 
is over the set $\Eff(\cZ)$ of degrees of maps from orbifold curves to
$\cZ$.

The genus-zero descendant potential $\cF^0_\cZ$ is a generating
function for all genus-zero Gromov--Witten invariants of $\cZ$.  It is
a formal power series in variables $t^a_k$, $0 \leq a \leq N$, $0
\leq k < \infty$, and the Novikov variables $Q_i$, $1 \leq i \leq
b_2(Z)$, defined by
\begin{equation}
  \label{eq:descendantpotentialwithQ}
  \cF^0_\cZ = \sum_{n \geq 0} \sum_{0 \leq k_1,\ldots,k_n < \infty}
  \sum_{d \in \Eff(\cZ)} {Q^d \over n!}
  \correlator{t_{k_1} \psi^{k_1},\ldots,t_{k_n} \psi^{k_n}}^\cZ_{0,n,d}
\end{equation}
where $t_k = t_k^\alpha \Phi_\alpha$.  The expression
$\correlator{t_{k_1} \psi^{k_1},\ldots,t_{k_n}
  \psi^{k_n}}^\cZ_{0,n,d}$ here expands, by multilinearity again,
into a polynomial in the variables $t^a_k$.

\subsection*{Analytic Continuation} Let us call the coefficient in
$F^0_\cZ$ of any monomial $\tau^{a_1}\cdots \tau^{a_n}$ a
\emph{coefficient series of $F^0_\cZ$}, and call the coefficient in
$\cF^0_\cZ$ of any monomial $t^{a_1}_{k_1} \cdots t^{a_n}_{k_n}$ a
\emph{coefficient series of $\cF^0_\cZ$}.  Each coefficient series is
a formal power series in the Novikov variables $Q_i$, $1 \leq i \leq
b_2(Z)$.  All of the examples we consider below satisfy:
\begin{itemize}
\item[(A)] each coefficient series of $F^0_\cZ$ converges in a
  neighbourhood of $Q_1 = Q_2 = \cdots = 0$ to an analytic function of
  the $Q_i$; and
\item[(B)] the coefficient series of $F^0_\cZ$ admit simultaneous
  analytic continuation to a neighbourhood of $Q_1 = Q_2 = \cdots =
  1$.
\end{itemize}
Condition (A) implies that each coefficient series of $\cF^0_\cZ$
converges in a neighbourhood of $Q_1 = Q_2 = \cdots = 0$ to an
analytic function of the $Q_i$, and condition (B) implies that
the coefficient series of $\cF^0_\cZ$ also admit simultaneous analytic
continuation to a neighbourhood of $Q_1 = Q_2 = \ldots = 1$: see
\cite{Coates--Ruan}*{Appendix} for discussion of a closely-related
point.

In what follows we will assume that a simultaneous analytic
continuation of the coefficient series has been chosen, and will set
$Q_1 = Q_2 = \ldots = 1$ throughout. Thus we regard the genus-zero
Gromov--Witten potential as a formal power series 
\begin{equation}
  \label{eq:GWpotentialwithoutQ}
  F^0_\cZ = \sum_{n \geq 0} 
  \sum_{d \in \Eff(\cZ)} {1 \over n!}
  \smcorrelator{\overbrace{\tau,\tau,\ldots,\tau}^{\text{$n$
        times}}}^\cZ_{0,n,d}, 
\end{equation}
in the variables $\tau^a$, $0 \leq a \leq N$, and we regard the
genus-zero descendant potential as a formal power series
\begin{equation}
  \label{eq:descendantpotentialwithoutQ}
  \cF^0_\cZ = \sum_{n \geq 0} \sum_{0 \leq k_1,\ldots,k_n < \infty}
  \sum_{d \in \Eff(\cZ)} {1 \over n!}
  \correlator{t_{k_1} \psi^{k_1},\ldots,t_{k_n} \psi^{k_n}}^\cZ_{0,n,d}
\end{equation}
in the variables $t^a_k$, $0 \leq a \leq N$, $0 \leq k < \infty$.

\subsection*{The Divisor Equation} 

The reader might worry that by suppressing Novikov variables ---
\emph{i.e.} by setting $Q_1 = Q_2 = \cdots = 1$ --- we have lost some
information about the degrees of curves.  This is not the case.  We
will discuss this for the case $\cZ = Y$; the case $\cZ = \cX$ is
entirely analogous.  Recall that our basis
$\varphi_0,\ldots,\varphi_N$ for $H(Y)$ was chosen so that
$\varphi_1,\ldots,\varphi_r$ is a lift to $T$-equivariant cohomology
of the basis $\omega_1',\ldots,\omega_r'$ for $H^2(Y;\CC)$ with which
we measure the degrees of curves.  Then, writing
\begin{align*}
  \tau = \tau^\alpha \varphi_\alpha, &&
  \tau_{\text{rest}} = \tau^0 \varphi_0 + \tau^{r+1} \varphi_{r+1} +
  \tau^{r+2} \varphi_{r+2} + \cdots + \tau^N \varphi_N,
\end{align*}
the Divisor Equation \cite{AGV:2}*{Theorem~8.3.1} gives
\[
F^0_Y = {1 \over 6} \big(\tau \cup \tau,\tau\big)_Y + 
\sum_{n \geq 0} 
\sum_{\substack{d \in \Eff(\cZ): \\ d \neq 0}} {e^{d_1 \tau^1} \cdots e^{d_r \tau^r} \over n!}
  \smcorrelator{\overbrace{\tau_{\text{rest}},\tau_{\text{rest}},\ldots,\tau_{\text{rest}}}^{\text{$n$
        times}}}^\cZ_{0,n,d}
\]
and so the substitution 
\[
e^{\tau^i} \longmapsto Q_i e^{\tau^i}, \qquad 1 \leq i \leq r,
\]
turns \eqref{eq:GWpotentialwithoutQ} into \eqref{eq:GWpotentialwithQ}.
The story for the descendant potential $\cF^0_Y$ is a little more
complicated but the upshot is the same: the Divisor Equation allows us
to recover \eqref{eq:descendantpotentialwithQ} from
\eqref{eq:descendantpotentialwithoutQ}.

\subsection*{The Lagrangian Submanifold-Germ}

Following Givental
\cite{Givental:quantization,Givental:symplectic,Coates--Givental:QRRLS}
we encode all genus-zero Gromov--Witten invariants of $\cZ$ via the
formal germ of a Lagrangian submanifold of $\cHZ$, defined as
follows.  Regard the genus-zero descendant potential $\cF^0_\cZ$ as the
formal germ of a function on $\cHZ^+$ via the change of variables
\[
q^a_k =
\begin{cases}
  t^0_1 - 1 & \text{if $(a,k) = (0,1)$} \\
  t^a_k & \text{otherwise}.
\end{cases}
\]
This change of variables is called the dilaton shift.  The variables
$q^a_k$ here are the Darboux co-ordinates from \eqref{eq:Darboux}, so
a general point on $\cHZ^+$ is $\sum_{k \geq 0} q_k^\alpha \Phi_\alpha
z^k$.  The dilaton shift makes $\cF^0_\cZ$ into the formal germ at
$-z$ of a function on $\cHZ^+$.  The graph of the differential of
$\cF^0_\cZ$ therefore defines the formal germ of a submanifold of
$\cHZ \cong T^\star \cHZ^+$, defined by the equations
\begin{align}
  \label{eq:coneequations}
  p^a_k = {\partial \cF^0_\cZ \over \partial q^a_k}
  && 0 \leq a \leq N, \; 0 \leq k < \infty.
\end{align}
We denote this Lagrangian submanifold-germ by $\cLZ$.

\subsection*{More Analytic Continuation}

In what follows we will need to analytically continue the
submanifold-germ $\cLZ$.  There is nothing exotic about this, as we
now explain.  The germ $\cLZ$ is defined by the equations
\eqref{eq:coneequations}, and to analytically continue $\cLZ$ we will
analytically continue each partial derivative ${\partial\cF^0_\cZ
  \over \partial q^a_k}$ in the variables\footnote{These variables
  correspond to basis elements of $H(\cZ)$ of degree $0$ or $2$.}
$t^a_0$, $0 \leq a \leq b_2(Z)$.  The partial derivative
${\partial\cF^0_\cZ \over \partial q^a_k}$ is a formal power series in
the variables $t^b_l$, $0 \leq b \leq N$, $0 \leq l < \infty$, so we
write it in the form
\begin{align*}
  \sum_I f_I\, t^I
  && \parbox{0.7\textwidth}{$t^I$ a monomial in the variables $t^b_l$
    with $b>b_2(Z)$ or $l>0$, \\
    $f_I$ a formal power series in the variables $t^a_0$, $0 \leq a
    \leq b_2(Z)$,}
\end{align*}
and then analytically continue each $f_I$.

\subsection*{The Crepant Resolution Conjecture} We are now ready to
state the conjecture.

\begin{conj} \label{CRC}
  There is a degree-preserving $\CC(\!(z^{-1})\!)$-linear
  symplectic isomorphism $\U:\cHX \to \cHY$ and a choice of analytic
  continuations of $\cLX$ and $\cLY$ such that $\U\(\cLX\) = \cLY$.
  Furthermore, $\U$ satisfies:
  \begin{itemize}
  \item[(a)] $\U(\fun_\cX) = \fun_Y + O(z^{-1})$;
  \item[(b)] $\U \circ \(\rho \CR\) = \({\pi^\star \rho} \, \cup\)
    \circ \U$ for every untwisted degree-two class $\rho \in
    H^2(\cX;\CC)$;
  \item[(c)] $\U\(\cHX^+ \)\oplus \cHY^- = \cHY$.
  \end{itemize}
\end{conj}

This is a slight modification of a conjecture due to
Coates, Corti, Iritani, and Tseng \cite{CCIT:crepant1}; very similar ideas
occurred, simultaneously and independently, in unpublished work of
Ruan.  An expository account of the conjecture and its consequences
can be found in \cite{Coates--Ruan}.  

\section{General Theory}
\label{sec:theory}
In this section we describe various aspects of Givental's symplectic
formalism which we will need below, as well as stating some
consequences of Conjecture~\ref{CRC}.

\subsection*{Big and Small $J$-Functions}

Let $\tau = \tau^\alpha \Phi_\alpha$.  The \emph{big $J$-function} of $\cZ$
is 
\[
\JJ_\cZ(\tau,z) := z + \tau + 
\sum_{n \geq 0} \sum_{d \in \Eff(\cZ)} 
{1 \over n!} 
\correlator{\tau,\tau,\ldots,\tau,{\Phi_\epsilon \over z -
    \psi}}^\cZ_{0,n+1,d} \Phi^\epsilon.
\]
It is a formal family of elements of $\cHZ$ --- in other words,
$\JJ_\cZ$ is a formal power series in the variables $\tau^a$, $0 \leq
a \leq N$, which takes values in $\cHZ$.  By writing out the equations
\eqref{eq:coneequations} defining $\cLZ$, it is easy to see that
$J_\cZ(\tau,-z)$ is the unique family of elements of $\cLZ$ of the
form $-z + \tau + O(z^{-1})$.

Take $\cZ = Y$ and restrict the parameter $\tau$ in the big
$J$-function to the locus $\tau = \tau^1 \varphi_1 + \cdots + \tau^r
\varphi_r$.  Then the Divisor Equation gives that $\JJ_Y(\tau^1 \varphi_1 +
\cdots + \tau^r \varphi_r,z)$ is equal to
\begin{equation}
  \label{eq:JYdivisor}
  z \, e^{\tau^1 \varphi_1/z} \cdots e^{\tau^r \varphi_r /z}
  \Bigg( 1 + \sum_{d \in \Eff(Y)} e^{d_1 \tau^1} \cdots e^{d_r \tau^r} 
  \correlator{\varphi_\epsilon \over z(z - \psi)}^Y_{0,1,d} \varphi^\epsilon \Bigg).
\end{equation}
Making the change of variables $q_i = e^{\tau^i}$, $1 \leq i \leq r$,
we define the \emph{small $J$-function} of $Y$ to be 
\begin{equation}
  \label{eq:JYsmall}
  J_Y(q,z) := z \, q_1^{\varphi_1/z} \cdots q_r^{\varphi_r /z}
  \Bigg( 1 + \sum_{d \in \Eff(Y)} q_1^{d_1} \cdots q_r^{d_r} 
  \correlator{\varphi_\epsilon \over z(z - \psi)}^Y_{0,1,d} \varphi^\epsilon \Bigg).
\end{equation}
In examples below we will see that this converges, in a domain where
each $|q_i|$ is sufficiently small, to a multi-valued analytic
function of $q_1,\ldots,q_r$ which takes values in $\cHY$.  The
multi-valuedness comes from the factors $q_i^{\varphi_i/z} := \exp(\varphi_i
\log(q_i)/z)$.  We have $J_Y(q,-z) \in \cLY$ for all $q$ in the domain
of convergence of $J_Y$.

Similarly, take $\cZ = \cX$ and restrict the parameter $\tau$ in the big
$J$-function to the locus $\tau = \tau^1 \phi_1 + \cdots + \tau^s
\phi_s$.  Then the Divisor Equation gives that 
\[
\JJ_\cX(\tau^1 \phi_1 +
\cdots + \tau^s \phi_s,z) = 
z \, e^{\tau^1 \phi_1/z} \cdots e^{\tau^s \phi_s /z}
\Bigg( 1 + \sum_{d \in \Eff(\cX)} e^{d_1 \tau^1} \cdots
e^{d_s \tau^s}  
\correlator{\phi_\epsilon \over z(z - \psi)}^\cX_{0,1,d} \phi^\epsilon \Bigg).
\]
Making the change of variables $u_i = e^{\tau^i}$, $1 \leq i \leq s$,
we define the \emph{small $J$-function} of $\cX$ to be
\begin{equation}
  \label{eq:JXsmall}
  J_\cX(u,z) := \\ z \, u_1^{\phi_1/z} \cdots u_s^{\phi_s /z}
  \Bigg( 1 + \sum_{d \in \Eff(\cX)} u_1^{d_1} \cdots u_s^{d_s} 
  \correlator{\phi_\epsilon \over z(z - \psi)}^\cX_{0,1,d} \phi^\epsilon \Bigg).
\end{equation}
In the examples below this converges, in a domain where each $|u_i|$
is sufficiently small, to a multi-valued analytic function of
$u_1,\ldots,u_s$ which takes values in $\cHX$.  We have $J_\cX(u,-z)
\in \cLX$ for all $u$ in the domain of convergence of $J_\cX$.

\subsection*{Two Consequences of Conjecture~\ref{CRC}}

Recall that the $T$-equivariant small quantum cohomology of $\cX$ is a
family of algebra structures on $H(\cX)$ parametrized by
$u_1,\ldots,u_s$, defined by
\begin{equation}
  \label{eq:smallQCX}
  \phi_\alpha \bullet \phi_\beta = 
  \sum_{d \in \Eff(\cX)} u_1^{d_1} \cdots u_s^{d_s} 
  \correlator{\phi_\alpha, \phi_\beta, \phi^\epsilon}^\cX_{0,3,d}
  \phi_\epsilon.
\end{equation}
The $T$-equivariant small quantum cohomology of $Y$ is a family of
algebra structures on $H(Y)$ parametrized by $q_1,\ldots,q_r$,
defined by
\begin{equation}
  \label{eq:smallQCY}
  \varphi_\alpha \bullet \varphi_\beta = 
  \sum_{d \in \Eff(Y)} q_1^{d_1} \cdots q_r^{d_r} 
  \correlator{\varphi_\alpha,\varphi_\beta,\varphi^\epsilon}^Y_{0,3,d} 
  \varphi_\epsilon.
\end{equation}
For the remainder of this subsection, assume that:
\begin{itemize}
\item Conjecture~\ref{CRC} holds;
\item the symplectic transformation $\U$ remains well-defined in the
  non-equivariant limit $\lambda \to 0$;
\item $\cX$ is semi-positive\footnote{The orbifold $\cZ$ is
      semi-positive if and only if there does not exist $d \in
      \Eff(\cZ)$ such that 
      \[
      3 - \dim_\CC \cZ \leq \langle c_1(T\cZ), d\rangle < 0
      \]
      All Fano and Calabi--Yau orbifolds are semi-positive, as are all
      orbifold curves, surfaces, and $3$-folds.  In particular, all
      the orbifolds that we consider in the examples below are
      semi-positive.}.
\end{itemize}
Two consequences of Conjecture~\ref{CRC} are then as follows: these
are proved\footnote{This is not, strictly speaking, true: the
  $T$-equivariant version of the Crepant Resolution Conjecture is not
  treated in \cite{Coates--Ruan}.  It is straightforward to check,
  however, that the arguments given there also prove the results
  stated here.  The key point is that $\U$ has a non-equivariant
  limit, and so only non-negative powers of $\lambda$ can occur.}  in
\cite{Coates--Ruan}.  Define the class $c \in H(Y)$ by
\[
\U(\fun_\cX) = \fun_Y - c z^{-1} + O(z^{-2}),
\]
and write 
\begin{align} \label{eq:defofci}
c = c^1 \varphi_1
+ \cdots + c^r \varphi_r + d \lambda, &&
c^1,\ldots,c^r,d \in \CC;
\end{align}
such an equality exists because $c$ has degree $2$.
Then:

\begin{cor} \label{cor:CCRC}
  The algebra obtained from the small quantum cohomology algebra of
  $Y$ by analytic continuation\footnote{\label{acfootnote}The analytic
    continuation of the small quantum product here is induced by the
    analytic continuation of $\cLY$.  This is explained in
    \cite{Coates--Ruan}.} in the parameters $q_{s+1},\ldots,q_r$ (if
  necessary) followed by the substitution
  \[
  q_i =
  \begin{cases}
    0 & 1 \leq i \leq s \\
    e^{c^i} & s < i \leq r
  \end{cases}
  \]
  is isomorphic to the Chen--Ruan orbifold cohomology algebra of
  $\cX$, via an isomorphism which sends $\alpha \in H^2(\cX;\CC)
  \subset H(\cX)$ to $\pi^\star \alpha \in H(Y)$.
\end{cor}

\noindent This is a version of Ruan's Cohomological Crepant Resolution
Conjecture \cite{Ruan:CCRC}.  

Define elements $b_e \in H(Y)$, $0 \leq
e\leq N$, by $b_e = 0$ if $\deg \phi_e \leq 2$ and 
\begin{align*}
  \U\big(\phi_e z^{1 - {1 \over 2} \deg \phi_e}\big) =
  b_e + O(z^{-1})
\end{align*}
otherwise.  Define power series $f^1,\ldots,f^r,g \in \CC[\![u_1,\ldots,u_s]\!]$ by
\begin{equation}
  \label{eq:defoffi}
  f^1 \varphi_1 + \cdots + f^r \varphi_r + g \lambda = 
  \sum_{d \in \Eff(\cX)}
  \sum_{e=r+1}^N
  (-1)^{{1 \over 2} \deg \phi_e +1}
  \correlator{\phi^e \psi^{{1 \over 2} \deg \phi_e - 2}}^\cX_{0,1,d}
  u_1^{d_1} \cdots u_s^{d_s} \, b_e;
\end{equation}
such an equality exists because each class $b_e$ has degree $2$.
Recall the definition of the rational numbers $r_i$, $1 \leq i \leq
s$, from Section~\ref{sec:statement}.  Then:

\begin{cor} \label{cor:Ruan}
  The algebra obtained from the small quantum cohomology algebra of
  $Y$ by analytic continuation\footnotemark[6] in the
  parameters $q_{s+1},\ldots,q_r$ (if necessary) followed by the
  substitution
  \[
  q_i =
  \begin{cases}
    e^{c^i + f^i} u_i^{r_i} & 1 \leq i \leq s \\
    e^{c^i + f^i} & s < i \leq r
  \end{cases}
  \]
  is isomorphic to the small quantum cohomology algebra of $\cX$, via
  an isomorphism which sends $\alpha \in H^2(\cX;\CC) \subset H(\cX)$
  to $\pi^\star \alpha \in H(Y)$.
\end{cor}

\noindent This is a ``quantum-corrected'' version of Ruan's Crepant
Resolution Conjecture.

\subsection*{Three Results Which We Will Need}

We next record three results which we will need below.  Part~(a)
follows from the String Equation: this is explained in \emph{e.g.}
\cite{Givental:symplectic}.  Part~(b) is a reconstruction result for
Gromov--Witten invariants --- it says that all genus-zero
Gromov--Witten invariants can be uniquely reconstructed from the
\emph{one-point descendants} $\correlator{\Phi_\alpha
  \psi^k}^\cZ_{0,1,d}$.  Part~(c) is a generalization of part~(b).
One can prove (b) and (c) by repeated application of the WDVV
equations and the Topological Recursion Relations; results along
similar lines can be found in \cite{Dubrovin,Kontsevich--Manin:GW,
  Bertram--Kley, Iritani:gen, Lee--Pandharipande, Rose}.

\begin{proposition} \ \label{thm:stuff}
  \begin{itemize}
  \item[(a)] The submanifold-germ $\cLZ \subset \cHZ$ is closed under
    multiplication by $\exp({a \lambda}/z)$ for any $a \in \CC$.
  \item[(b)] If $\cZ$ is semi-positive and the Chen--Ruan
    orbifold cohomology algebra of $\cZ$ is generated by
    $H^2(\cZ;\CC)$ then the submanifold-germ $\cLZ$ can be uniquely
    reconstructed from the small $J$-function $J_\cZ(q,z)$.
  \item[(c)] If $\cZ$ is semi-positive and $H^2_{\text{gen}} \subset
    H^2_{\text{CR}}(\cZ;\CC)$ is a subspace such that the Chen--Ruan
    orbifold cohomology algebra of $\cZ$ is generated by
    $H^2_{\text{gen}}$ then the submanifold-germ $\cLZ$ can be
    uniquely reconstructed from the restriction of the big
    $J$-function $\JJ_\cZ(\tau,z)$ to the locus $\tau \in
    H^2_{\text{gen}}$.  \qed
  \end{itemize}
\end{proposition}

\noindent It is easy to check that in all the examples we consider
below, the Chen--Ruan cohomology algebra of $\cZ$ is generated in
degree $2$.

\subsection*{Computing Twisted Gromov--Witten Invariants}

As discussed above, in most of our examples $\cZ$ will be the total
space of a concave vector bundle $\cE$ over a compact orbifold $\cB$,
and the $T$-action on $\cZ$ will be the canonical $\Cstar$-action
which rotates the fibers of $\cE$ and covers the trivial action on
$\cB$.  In this situation $\Eff(\cZ)$ is canonically isomorphic to
$\Eff(\cB)$ and $H(\cZ)$ is canonically isomorphic to $H(\cB) :=
H^\bullet_{\text{CR}}(\cB;\CC) \otimes \CC(\lambda)$. Our bases
$\{\Phi_a\}$ and $\{\Phi^a\}$ for $H(\cZ)$ determine bases for
$H(\cB)$, which we also denote by $\{\Phi_a\}$ and $\{\Phi^a\}$.
Gromov--Witten invariants of $\cZ$ coincide with Gromov--Witten
invariants of $\cZ$ twisted, in the sense of
\cite{Coates--Givental:QRRLS, CCIT:computing}, by the $T$-equivariant
inverse Euler class $\be^{-1}$ and the vector bundle $\cE$.  Results
in \cite{CCIT:computing} allow the calculation of twisted
Gromov--Witten invariants in a quite general setting.  We will need
three special cases of these results, as follows.  Each of these
special cases determines a family of elements $q \mapsto I_\cZ(q,-z)$
of elements of $\cLZ$; in each case this family $I_\cZ(q,z)$ is an
appropriate \emph{hypergeometric modification} of the small
$J$-function $J_\cB(q,z)$ of $\cB$.

\begin{thm} \label{thm:smalllinebundle} Suppose that $\cE \to \cB$ is
  a concave line bundle.  Let $\rho$ denote the first Chern class of
  $\cE$, regarded as an element of localized $T$-equivariant
  Chen--Ruan cohomology $H(\cB)$ via the canonical inclusion
  $H^\bullet(\cB;\CC) \hookrightarrow H^\bullet_{\text{CR}}(\cB;\CC)$, and
  set
  \[
  M_\cE(d) := \prod_{\substack{b: \langle \rho,d \rangle < b \leq 0,\\
      \fr(b) = \fr(\langle \rho,d \rangle)}}
  (\lambda + \rho + b z)
  \]
  where $d \in \Eff(\cB)$ and $\fr(r)$ denotes the fractional part of
  $r$.  Let $k = b_2(\cB)$, so that the small $J$-function of $\cB$ is
  \[
  J_\cB(q,z) =  z \, q_1^{\Phi_1/z} \cdots q_k^{\Phi_k /z}
  \Bigg( 1 + \sum_{d \in \Eff(\cB)} q_1^{d_1} \cdots q_k^{d_k}  
  \correlator{\Phi_\epsilon \over z(z - \psi)}^\cB_{0,1,d}
  \Phi^\epsilon \Bigg). 
  \]
  Then
  \begin{equation}
    \label{eq:defofIZ}
    I_\cZ(q,z) :=  z \, q_1^{\Phi_1/z} \cdots q_k^{\Phi_k /z}
    \Bigg( 1 + \sum_{d \in \Eff(\cB)} 
    q_1^{d_1} \cdots q_k^{d_k} \, 
    M_\cE(d) \,
    \correlator{\Phi_\epsilon \over z(z - \psi)}^\cB_{0,1,d}
    \Phi^\epsilon \Bigg)
  \end{equation}
  satisfies $I_\cZ(q,-z) \in \cLZ$ for all $q$ in the domain of
  convergence of $I_\cZ$.
\end{thm}
 
\begin{proof}
  Theorem~4.6 in \cite{CCIT:computing} concerns a Lagrangian
  submanifold-germ $\cL^\tw$ which encodes twisted Gromov--Witten
  invariants: in our situation, $\cL^\tw = \cL_\cZ$.  The Theorem
  gives a formula for a formal family $\tau \mapsto I^\tw(\tau,-z)$ of
  elements of $\cL^\tw$, as follows.  Let $\cI$ be a set which indexes
  the components of the inertia stack $\cI\cB$ of $\cB$, and let $0
  \in \cI$ be the index of the distinguished component $\cB \subset
  \cI\cB$.  One decomposes the big $J$-function of $\cB$ as a sum
  \[
  \JJ_\cB(\tau,z) = \sum_{\theta \in \NETT(\cB)} J_\theta(\tau,z)
  \]
  of contributions from stable maps of different \emph{topological
    types}; here $\NETT(\cB)$ is the set of topological types.  The
  topological type of a degree-$d$ stable map $f:\cC \to \cB$ from a
  genus-$g$ orbifold curve with $n$ marked points is the triple
  $(g,d,S)$, where $S = (i_1,\ldots,i_n)$ is the ordered $n$-tuple of
  elements of $\cI$ indexing the components of $\cI\cB$ picked out by
  the marked points.  Then
  \[
  I^\tw(\tau,z) := \sum_{\theta \in \NETT(\cB)} M_\theta(z) \cdot
  J_\theta(\tau,z)
  \]
  where $M_\theta(z)$ is a \emph{modification factor} defined in \S4.2
  of \cite{CCIT:computing}.

  If we set $\tau = \tau^1 \Phi_1 + \cdots + \tau^k \Phi_k$ then
  $J_\theta(\tau,z)$ vanishes unless the topological type $\theta$ is
  of the form $(0,d,S)$ where $S = (0,0,\ldots,0,i)$ for some $i \in
  \cI$; this is because the classes $\Phi_i$, $1 \leq i \leq k$ are
  supported on the distinguished component $\cB$ of $\cI\cB$.  In this
  case the modification factor $M_\theta(z)$ depends only on $d$ and
  is equal to $M_\cE(d)$.  Also,
  \[
  \JJ_\cB(\tau^1 \Phi_1 +
  \cdots + \tau^k \Phi_k,z) = 
  z \, e^{\tau^1 \Phi_1/z} \cdots e^{\tau^k \Phi_k /z}
  \Bigg( 1 + \sum_{d \in \Eff(\cB)} e^{d_1 \tau^1}
  \cdots e^{d_k \tau^k}   
  \correlator{\Phi_\epsilon \over z(z - \psi)}^\cB_{0,1,d}
  \Phi^\epsilon \Bigg)
  \]
  and it follows that $I^\tw(\tau^1 \Phi_1 + \cdots + \tau^k
  \Phi_k,z)$ is equal to
  \[
  z \, e^{\tau^1 \Phi_1/z} \cdots e^{\tau^k \Phi_k /z}
  \Bigg( 1 + \sum_{d \in \Eff(\cB)} e^{d_1 \tau^1}
  \cdots e^{d_k \tau^k}\, M_\cE(d)\,   
  \correlator{\Phi_\epsilon \over z(z - \psi)}^\cB_{0,1,d}
  \Phi^\epsilon \Bigg).
  \]
  Making the change of variables $q_i = e^{\tau^i}$, $1 \leq i \leq
  k$, we conclude that $I_\cZ(q,-z) \in \cLZ$ for all $q$ such that
  the series defining $I_\cZ$ converges.
\end{proof}

Exactly the same argument proves:

\begin{thm} \label{thm:smallvb}
  If $\cE = \cE_1 \oplus \cdots \oplus \cE_m$ is the direct sum of
  convex line bundles, 
  \[
  M_\cE(d) := \prod_{1 \leq i \leq m}
  M_{\cE_i}(d),
  \] 
  and $I_\cZ(q,z)$ is defined exactly as in \eqref{eq:defofIZ} then
  $I_\cZ(q,-z) \in \cLZ$ for all $q$ in the domain of convergence of
  $I_\cZ$. \qed
\end{thm}

The final special case which we need is where $\cZ$ is the total space
of a direct sum of line bundles $\cE = \cE_1 \oplus \cdots \oplus
\cE_m$ over $\cB = B\ZZ_n$.  Components of the inertia stack of
$B\ZZ_n$ are indexed by fractions $k/n$, $0 \leq k < n$: the
component indexed by $k/n$ corresponds to the element $[k] \in
\ZZ_n$.  Let $\fun_{k/n} \in H(\cB)$ denote the orbifold
cohomology class which restricts to the unit class on the component of
the inertia stack indexed by $k/n$ and restricts to zero on
the other components.  The set $\{\fun_{k/n} : 0 \leq k < n\}$
forms a basis for $H(\cB)$; as $H(\cB)$ and $H(\cZ)$ are
canonically isomorphic it determines a basis for $H(\cZ)$ as well.

\begin{thm} \label{thm:BZn}
  Let $\cZ$ be the total space of the direct sum of line bundles $\cE
  = \cE_1 \oplus \cdots \oplus \cE_m$ over $\cB = B\ZZ_n$.  Let $e_i$
  be the integer such that $\cE_i$ is given by the character $[k]
  \mapsto \exp({2 \pi \tti e_i k  \over n})$ of $\ZZ_n$ and that
  $0 \leq e_i < n$.  Let
  \[
  P_{i,k} := \Big\{ b : \fr(b) = \fr\big({\textstyle-{e_i k \over
      n}}\big),  \,
  {\textstyle-{e_i k \over n}} < b \leq 0\Big\}
  \]
  and 
  \[
  I_\cZ(x,z) := \sum_{k \geq 0} 
  x^k { \prod_{i=1}^m \prod_{b \in P_{i,k}}
  \big({e_i \over n}\lambda + b z\big) \over k! \, z^k} \fun_{\fr({k/n})}.
  \]
  Then $x \mapsto I_\cZ(x,-z)$ is a formal family of elements of
  $\cLZ$.
\end{thm}

\begin{proof}
  We argue as in the proof of Theorem~\ref{thm:smalllinebundle}.  If
  we decompose the big $J$-function of $\cB$ as a sum
  \[
  \JJ_\cB(\tau,z) = \sum_{\theta \in \NETT(\cB)} J_\theta(\tau,z)
  \]
  of contributions from stable maps of different topological types and
  set
  \[
  I^\tw(\tau,z) := \sum_{\theta \in \NETT(\cB)} M_\theta(z) \cdot
  J_\theta(\tau,z)
  \]
  where $M_\theta(z)$ is defined in \cite{CCIT:computing}*{\S4.2}
  then $\tau \mapsto I^\tw(\tau,-z)$ defines a formal family of
  elements of $\cL^\tw = \cLZ$.  Proposition~6.1 in
  \cite{CCIT:computing} gives an explicit formula for the big
  $J$-function of $\cB = B \ZZ_n$, and we see from this that if $\tau = x
  \fun_{1 \over n}$ then $J_\theta(\tau,z)$ vanishes unless the
  topological type $\theta$ is $(0,0,S)$ with
  \[
  S = \Big(\overbrace{\textstyle{1 \over n},{1 \over n},\ldots,{1 \over
      n}}^{\text{$k$ times}},\fr \big(\textstyle{n-k \over
    n} \big) \Big).
  \]
  In this case, 
  \begin{align*}
    J_\theta(\tau,z) = {x^k \over k! \, z^k} \fun_{\fr({k/n})}
    && \text{and} &&
    M_\theta(z) = \prod_{i=1}^m \prod_{b \in P_{i,k}} \big(\textstyle{e_i
      \over n} \lambda + b z\big). 
  \end{align*}
  Thus $x \mapsto I_\cZ(x,-z)$ is a formal family of elements of
  $\cLZ$.
\end{proof}

\section{Example I: $\cX = \big[\CC^3/\ZZ_3\big]$, $Y = K_{\PP^2}$}
\label{sec:C3Z3}

Let $\cX$ be the orbifold $\big[ \CC^3/\ZZ_3 \big]$ where $\ZZ_3$ acts
on $\CC^3$ with weights $(1,1,1)$.  The coarse moduli space $X$ of
$\cX$ is the quotient singularity $\CC^3/\ZZ_3$ and the crepant
resolution $Y$ of $X$ is the canonical bundle $K_{\PP^2}$.

\subsection*{Toric Geometry}

Let $e_1$, $e_2$, $e_3$ denote the standard basis vectors for $\RR^3$.
The space $Y$ is the toric variety corresponding to a fan with rays
\begin{equation}
  \label{eq:KP2rays}
  \text{$e_1+e_3$, $e_2+e_3$, ${-e_1}-e_2+e_3$, $e_3$;}
\end{equation}
this fan is a cone over the picture in the plane $z=1$ shown in
Figure~\ref{fig:KP2fan}. 

\begin{figure}[bhtp]
  \centering
  \subfloat[$\cX$]{
    \begin{picture}(80,80)(-40,-40)
      \multiput(-30,-30)(30,0){3}{\makebox(0,0){$\cdot$}}
      \multiput(-30,0)(30,0){3}{\makebox(0,0){$\cdot$}}
      \multiput(-30,30)(30,0){3}{\makebox(0,0){$\cdot$}}
      \put(30,0){\makebox(0,0){$\bullet$}}
      \put(0,30){\makebox(0,0){$\bullet$}}
      \put(-30,-30){\makebox(0,0){$\bullet$}}
      \put(36,6){\makebox(0,0){$\scriptstyle 1$}}
      \put(6,36){\makebox(0,0){$\scriptstyle 2$}}
      \put(-36,-24){\makebox(0,0){$\scriptstyle 3$}}
      \put(30,0){\line(-1,1){30}}
      \put(-30,-30){\line(1,2){30}}
      \put(-30,-30){\line(2,1){60}}
    \end{picture}
    \label{fig:C3Z3fan}
  }
  \qquad \qquad
  \subfloat[$Y$]{
      \begin{picture}(80,80)(-40,-40)
      \multiput(-30,-30)(30,0){3}{\makebox(0,0){$\cdot$}}
      \multiput(-30,0)(30,0){3}{\makebox(0,0){$\cdot$}}
      \multiput(-30,30)(30,0){3}{\makebox(0,0){$\cdot$}}
      \put(30,0){\makebox(0,0){$\bullet$}}
      \put(0,30){\makebox(0,0){$\bullet$}}
      \put(-30,-30){\makebox(0,0){$\bullet$}}
      \put(0,0){\makebox(0,0){$\bullet$}}
      \put(36,6){\makebox(0,0){$\scriptstyle 1$}}
      \put(6,36){\makebox(0,0){$\scriptstyle 2$}}
      \put(-36,-24){\makebox(0,0){$\scriptstyle 3$}}
      \put(6,6){\makebox(0,0){$\scriptstyle 4$}}
      \put(30,0){\line(-1,1){30}}
      \put(-30,-30){\line(1,2){30}}
      \put(-30,-30){\line(2,1){60}}
      \put(0,0){\line(1,0){30}}
      \put(0,0){\line(0,1){30}}
      \put(0,0){\line(-1,-1){30}}
    \end{picture}
    }
    \caption{The fans for $\cX$ and $Y$ (respectively) are the cones over these
      pictures in the plane $z=1$}
  \label{fig:KP2fan}
\end{figure}
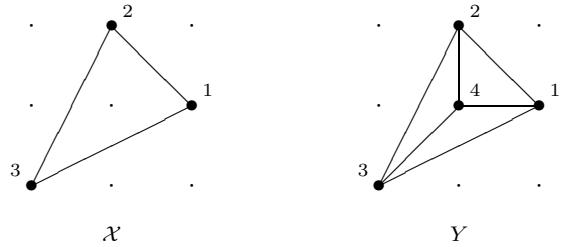
We can construct $Y$ as a GIT quotient, following \emph{e.g.}
\cite{Audin}, by considering the exact sequence
\begin{equation}
  \label{eq:KP2sequence}
  \begin{CD}
    0 @>>> \ZZ @>{
      \begin{pmatrix}
        1 \\ 1 \\ 1 \\ -3
      \end{pmatrix}}>>  
    \ZZ^4 @>{
      \begin{pmatrix} \textstyle
        1 & 0 & -1 & 0 \\
        0 & 1 & -1 & 0 \\
        1 & 1 &  1 & 1
      \end{pmatrix}}>> \ZZ^3 @>>> 0.
  \end{CD}.
\end{equation}
This shows that $Y$ is a quotient $\CC^4/\!\!/\Cstar$, where $\tau \in
\Cstar$ acts on $\CC^4$ as 
\begin{equation}
  \label{eq:KP2action}
  (x, y, z, w) \longmapsto
  (\tau x,  \tau y,   \tau z,   \tau^{-3} w)
\end{equation}
Dualizing \eqref{eq:KP2sequence} gives
\[
\begin{CD}
  0 @>>> \ZZ^3 @>{
    \begin{pmatrix} \textstyle
      1 & 0 & 0 \\
      0 & 1 & 0 \\
      -1 & -1 & 3 \\
      0 & 0 & 1
    \end{pmatrix}}>>
  \ZZ^4 @>{
\begin{pmatrix}
      1 & 1 & 1 & -3
    \end{pmatrix}
    }>> \ZZ @>>> 0
\end{CD}
\]
where the right-hand entry is $H^2(Y;\ZZ)$ and the columns of the
right-hand matrix give the four toric divisors in $Y$.  If we draw
this picture in $H^2(Y;\RR)$ then it gives the chamber decomposition
for the GIT problem (Figure~\ref{fig:KP2secondaryfan} below); this
chamber decomposition is also known as the \emph{secondary fan} for
$Y$.
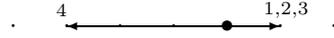
\begin{figure}[hbtp]
  \centering
     \begin{picture}(160,10)(-90,-5)
      \multiput(-80,0)(20,0){7}{\makebox(0,0){$\cdot$}}
      \put(0,0){\makebox(0,0){$\bullet$}}
      \put(22,6){\makebox(0,0){$\scriptstyle 1,2,3$}}
      \put(-62,6){\makebox(0,0){$\scriptstyle 4$}}
      \put(0,0){\vector(1,0){20}}
      \put(0,0){\vector(-1,0){60}}
    \end{picture}
    \caption{The secondary fan for $Y = K_{\PP^2}$}
  \label{fig:KP2secondaryfan}
\end{figure}

Each chamber in the secondary fan corresponds to a fan $\Sigma$ which
is a triangulation of the rays \eqref{eq:KP2rays}: a cone $\sigma$ is
in $\Sigma$ if and only if the co-ordinate subspace corresponding to
the complement of $\sigma$ covers the chosen chamber.  The fans are
shown in Figure~\ref{fig:KP2fan}.  For $\xi$ in the left-hand chamber
the GIT quotient $\CC^4 \GIT{\xi}\Cstar$ gives $\cX$; we delete the
locus $w=0$ from $\CC^4$ and then take the quotient by the action
\eqref{eq:KP2action}.  For $\xi$ in the right-hand chamber we have
$\CC^4 \GIT{\xi} \Cstar = Y$; we delete the locus $(x,y,z) = (0,0,0)$
from $\CC^4$ and then take the quotient by \eqref{eq:KP2action}.  For
$\xi = 0$ the quotient $\CC^4 \GIT{\xi} \Cstar$ gives the coarse
moduli space $X$.  Moving from the right-hand chamber into the
``wall'' $\xi = 0$ gives the resolution map $Y \to X$; this sends
\begin{align*}
  [x,y,z,w]
  \in \CC^4 \GIT{\xi} \Cstar
  && \text{to} &&
  \big[ x w^{1/3},  y w^{1/3}, z w^{1/3} \, \big]
  \in \CC^3/\ZZ_3
\end{align*}
where $[A]$ denotes class of $A$ in the appropriate quotient.

\subsection*{The $T$-Action}

Consider the action of $T = \Cstar$ on $\CC^4$ such that $\alpha \in
T$ acts as
\[
(x,y,z,w) 
\longmapsto
(x,y,z,\alpha w)
\]
This action descends to give $T$-actions on $\cX$, $X$, and $Y$.
The induced action on $\cX$ is 
\[
[x,y,z] \longmapsto
\big[ \alpha^{1/3} x, \alpha^{1/3} y, \alpha^{1/3} z \, \big]
\]
The induced action on $Y$ is the canonical $\Cstar$-action on the line
bundle $K_{\PP^2} \to \PP^2$; it covers the trivial action on
$\PP^2$.  The diagram
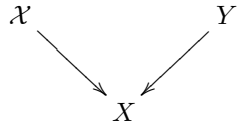
\[
\xymatrix{
  \cX \ar[rd] & & Y \ar[ld]\\
  &X&} 
\]
is $T$-equivariant.

\subsection*{Bases for Everything}

We have
\begin{align*}
  & r := \rank H^2(Y;\CC) = 1, &
  & s := \rank H^2(\cX;\CC) = 0.
\end{align*}
Let $p$ be the first Chern class of the line bundle $\cO(1) \to
\PP^2$, pulled back to $Y = K_{\PP^2}$ via the projection $K_{\PP^2}
\to \PP^2$.  The class $p$ has a canonical lift to $T$-equivariant
cohomology, which we also denote by $p$, and
\[
H(Y) = \CC(\lambda)[p]/\langle p^3 \rangle.
\]
We set
\begin{align*}
  & \varphi_0 = 1, &
  &  \varphi_1 = p, &
  & \varphi_2 = p^2, \\
  & \varphi^0 = \lambda p^2, &
  & \varphi^1 = \lambda p - 3 p^2, &
  & \varphi^2 = \lambda - 3 p.  
\end{align*}

The components of the inertia stack of $\cX$ are indexed by elements
of $\ZZ_3$.  Let $\fun_{k/3} \in H(\cX)$ denote the orbifold
cohomology class which restricts to the unit class on the inertia
component indexed by $[k] \in \ZZ_3$ and restricts to zero on the
other components.  Set
\begin{align*}
  &\phi_0 = \fun_0, &
  &\phi_1 = \fun_{1/3}, &
  &\phi_2 = \fun_{2/3}, \\
  & \phi^0 = \textstyle {\lambda^3 \over 9} \fun_0, &
  & \phi^1 = 3 \fun_{2/3}, &
  & \phi^2 = 3 \fun_{1/3}.  
\end{align*}

\subsection*{Step 1: A Family of Elements of $\cLY$}

Consider
\begin{equation}
  \label{eq:KP2IGamma}
  I_Y(y,z) := z \sum_{d \geq 0}
  { \Gamma\big(1 + {p \over z} \big)^3 
    \over \Gamma\big(1 + {p \over z} + d \big)^3}
  { \Gamma\big(1 + {\lambda - 3 p \over z} \big)
    \over \Gamma\big(1 + {\lambda - 3 p \over z} - 3d \big)} \, 
  y^{d + p/z}.
\end{equation}
This series converges in the region $\big\{y \in \CC : 0 < |y| < {1
  \over 27} \big\}$ to a multi-valued analytic function of $y$ which
takes values in $\cHY$.  We have
\begin{equation}
  \label{eq:KP2I}
  I_Y(y,z) = z \sum_{d \geq 0}
  {
    \prod_{-3d < m \leq 0} (\lambda - 3 p + m z)
    \over \prod_{0 < m \leq d} (p + m z)^3} \,
  y^{d + p/z}.
\end{equation}

\begin{proposition} \label{thm:KP2mirror}
  \begin{align*}
    I_Y(y,-z) \in \cLY && \text{for all $y$ such that $0<|y|<{1 \over
        27}$}. 
  \end{align*}
\end{proposition}

\begin{proof}
  We are in the situation of Theorem~\ref{thm:smalllinebundle} with
  $\cB = \PP^2$ and $\cE = \cO(-3)$.  Givental has proved
  \cite{Givental:equivariant} that the small $J$-function of $\PP^2$
  is
  \[
  J_{\PP^2}(q,z) = z \, q^{p/z} \sum_{d \geq 0}
  { q^{d} \over \prod_{0 < m \leq d} (p+mz)^3},
  \]
  and it follows (by comparing with the statement of
  Theorem~\ref{thm:smalllinebundle}) that
  \begin{align}
    \label{eq:P2Jpart}
    \correlator{{\Phi^\epsilon \over z - \psi}}^{\PP^2}_{0,1,d}
    \Phi_\epsilon = {1 \over \prod_{0 < m \leq d} (p+mz)^3}
    && \text{whenever $d>0$.}
  \end{align}
  Theorem~\ref{thm:smalllinebundle} thus implies that $I_Y(y,-z) \in
  \cLY$ for all $y$ in the domain of convergence of $I_Y$, as claimed.
\end{proof}

\subsection*{Step 2: $I_Y$ Determines $\cLY$}

We have:

\begin{cor} \label{cor:KP2mirror}
  \begin{align*}
    J_Y(q,z) &= e^{\lambda f(y)/z} I_Y(y,z) &
    \text{where} &&
    q & = y \exp \big(3 f(y)\big), \\
    && && f(y) &= \sum_{d>0} \textstyle{(3d-1)! \over (d!)^3} (-y)^d.
  \end{align*}
\end{cor}

\begin{proof}
  We have
  \[
  I_Y(y,z) = z + p \log y - (\lambda - 3 p)f(y) + O(z^{-1}).
  \]
  Applying Propositions~\ref{thm:stuff} and~\ref{thm:KP2mirror}, we
  see that
  \begin{align*}
    y \mapsto e^{{-\lambda} f(y)/z} I_Y(y,-z), && 0 < |y| <
    \textstyle {1 \over 27}, 
  \end{align*}
  is a family of elements of $\cLY$.  But
  \[
  e^{-\lambda f(y)/z} I_Y(y,-z) = -z + p \log q + O(z^{-1}),
  \]
  where $q$ is defined above, and the unique family of elements of
  $\cLY$ of this form is $q \mapsto J_Y(q,-z)$.
\end{proof}

As $I_Y(y,z)$ is multivalued-analytic and the change of variables $y
\rightsquigarrow q$ is analytic, we conclude that the series defining
$J_Y(q,z)$ converges, when $|q|$ is sufficiently small, to a
multivalued analytic function of $q$.  Furthermore, as the small
$J$-function $J_Y(q,z)$ determines $\cLY$
(Proposition~\ref{thm:stuff}b), it follows that $\cLY$ is uniquely
determined by the fact that $y \mapsto I_Y(y,-z)$ is a family of
elements of $\cLY$.

\subsection*{Aside: Computing Gromov--Witten Invariants of $Y$}

As is well-known, Corollary~\ref{cor:KP2mirror} determines many
genus-zero Gromov--Witten invariants of $Y$.  The inverse to the
change of variables $y \rightsquigarrow q$ is
\[
y = q + 6q^2 + 9q^3 + 56q^4 - 300q^5+ \ldots
\] 
Substituting this into the equality
\[
z \, q^{p/z}
\Bigg( 1 + \sum_{d > 0} q^d
\correlator{\varphi_\epsilon \over z(z - \psi)}^Y_{0,1,d}
\varphi^\epsilon \Bigg) = 
z \, e^{\lambda f(y)/z} \sum_{d \geq 0}
{
  \prod_{-3d < m \leq 0} (\lambda - 3 p + m z)
  \over \prod_{0 < m \leq d} (p + m z)^3} \,
y^{d + p/z}
\]
and comparing coefficients of $q$, one finds that
\begin{align*}
  \correlator{\varphi^\alpha \over z - \psi}^Y_{0,1,1}
  \varphi_\alpha &= -{9p^2 \over z} + o(\lambda),
  & \correlator{\varphi^\alpha \over z - \psi}^Y_{0,1,2}
  \varphi_\alpha &= {135 p^2 \over 4z} + o(\lambda), \\
  \correlator{\varphi^\alpha \over z - \psi}^Y_{0,1,3}
  \varphi_\alpha &= -{244p^2 \over z} + o(\lambda), &
  \correlator{\varphi^\alpha \over z - \psi}^Y_{0,1,4}
  \varphi_\alpha &= {36999 p^2 \over 16z} + o(\lambda),
\end{align*}
and so on, where $o(\lambda)$ denotes terms containing strictly
positive powers of $\lambda$.  Taking the non-equivariant limit
$\lambda \to 0$ yields the local Gromov--Witten invariants $K_d$
calculated in \cite{Chiang--Klemm--Yau--Zaslow}*{\S2.2}:
\begin{align*}
  \correlator{\vphantom{\big\vert}p}^Y_{0,1,1}
  &= 3,
  & \correlator{\vphantom{\big\vert}p}^Y_{0,1,2}
  &= -{45 \over 4}, \\
  \correlator{\vphantom{\big\vert}p}^Y_{0,1,3}
  &= {244 \over 3}, &
  \correlator{\vphantom{\big\vert}p}^Y_{0,1,4}
  &= -{12333 \over 16},
\end{align*}
and therefore, using the Divisor Equation, we find
\begin{align*}
  K_1 &= 3, &
  K_2 &= {-{45 \over 8}}, &
  K_3 &= {244 \over 9}, &
  K_4 &= {-{12333 \over 64}}, &
  \text{etc.}
\end{align*}

\subsection*{Step 3: A Family of Elements of $\cLX$}

Let
\begin{equation}
  \label{eq:IC3Z3}
  I_\cX(x,z) := z \, x^{-\lambda/z}
  \sum_{l \geq 0} {x^l  \over l! \, z^l}
  \prod_{\substack{b : 0 \leq b < {l  \over 3} \\ \fr{b} = \fr{l
        \over 3}}} 
  \big(\textstyle {\lambda \over 3} - b z)^3 \;
  \fun_{\fr({l\over 3})}.
\end{equation}
This converges, in the region $|x|<27$, to an analytic function which
takes values in $\cHX$.  Theorem~\ref{thm:BZn} and
Proposition~\ref{thm:stuff}(a) imply that $I_\cX(x,-z) \in \cLX$ for
all $x$ such that $|x|<27$.

\subsection*{Step 4: $I_\cX$ Determines $\cLX$}

We have:

\begin{cor} \label{cor:C3Z3mirror}
  \begin{align*}
    \JJ_\cX\big(\tau^1 \fun_{1/3},z\big) &= x^{\lambda/z} I_\cX(x,z) & \text{where}
    & & \tau^1 &= \sum_{m \geq 0} (-1)^m {x^{3m+1} \over (3m+1)!}
    {\Gamma\big(m+\textstyle{1 \over 3}\big)^3 \over
      \Gamma\big(\textstyle{1 \over 3}\big)^3}.
  \end{align*}
\end{cor}

\begin{proof}
  On the one hand, we know that $ x^{-\lambda/z} I_\cX(x,-z) \in \cLX$,
  and on the other hand we know that
  \[
  x^{-\lambda/z} I_\cX(x,-z) = -z + \left(\sum_{m \geq 0} (-1)^m {x^{3m+1} \over (3m+1)!}
  {\Gamma\big(m+\textstyle{1 \over 3}\big)^3 \over
    \Gamma\big(\textstyle{1 \over 3}\big)^3}\right) \fun_{1/3} +
  O(z^{-1}).
  \]
  As the unique family of elements of $\cLX$ of the form $-z + \tau^1
  \fun_{1/3} + O(z^{-1})$ is $\tau^1 \mapsto \JJ_\cX(\tau^1
  \fun_{1/3},-z)$, the result follows.
\end{proof}

Since $x^{\lambda/z} I_\cX(x,z)$ and the change of variables $x
\rightsquigarrow \tau^1$ are analytic, this implies that
$\JJ_\cX\big(\tau^1 \fun_{1/3},z\big)$ depends analytically on
$\tau^1$ in some region where $|\tau^1|$ is sufficiently small.  It
also, via Proposition~\ref{thm:stuff}c, shows that $\cLX$ is uniquely
determined by the fact that $x \mapsto I_\cX(x,-z)$ is a family of
elements of $\cLX$.

\subsection*{Aside: Computing Gromov--Witten Invariants of $\cX$}

Just as we did for $Y$, one can use Corollary~\ref{cor:C3Z3mirror} to
compute genus-zero Gromov--Witten invariants of $\cX$.  This
calculation is carried out in \cite{CCIT:computing}*{\S6.3}; it
verifies some of the predictions made by Aganagic, Bouchard, and Klemm
\cite{ABK}.

\subsection*{Step 5: The $B$-model Moduli Space and the Picard--Fuchs
  System}

The \emph{$B$-model moduli space} $\cM_B$ is the toric
orbifold corresponding to the secondary fan for $Y$.
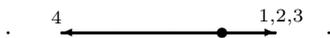
\begin{figure}[h]
  \centering
     \begin{picture}(160,10)(-90,-5)
      \multiput(-80,0)(20,0){7}{\makebox(0,0){$\cdot$}}
      \put(0,0){\makebox(0,0){$\bullet$}}
      \put(22,6){\makebox(0,0){$\scriptstyle 1,2,3$}}
      \put(-62,6){\makebox(0,0){$\scriptstyle 4$}}
      \put(0,0){\vector(1,0){20}}
      \put(0,0){\vector(-1,0){60}}
    \end{picture}
    \caption{The secondary fan for $Y = K_{\PP^2}$.}
\end{figure}
It has two co-ordinate patches, one for each chamber.  Let $x$ be the
co-ordinate corresponding to the left-hand chamber (recall that this
chamber gives rise to $\cX$) and let $y$ be the co-ordinate
corresponding to the right-hand chamber (recall that this chamber
gives $Y$).  The co-ordinate patches are related by
\begin{equation}
  \label{eq:KP2gluing}
  y = x^{-3}
\end{equation}
and it follows that $\cM_B$ is the weighted projective space
$\PP(1,3)$.  The space $\cM_B$ is called the $B$-model moduli space as
it is the base of the Landau--Ginzburg model (``the $B$-model'') which
corresponds to the quantum cohomology of $Y$ (``the $A$-model'') under
mirror symmetry: see \emph{e.g.}
\citelist{\cite{Givental:toric}\cite{Hori--Vafa}}.

We regard $I_\cX(x,z)$ as a function on the co-ordinate patch
corresponding to $\cX$ and $I_Y(y,z)$ as a function on the co-ordinate
patch corresponding to $Y$.  Writing
\[
I_Y(y,z) = 
I^0_Y\, \varphi_0 +
I^1_Y\, \varphi_1 +
I^2_Y\, \varphi_2, 
\]
the components $\big\{I^j_Y: j=0,1,2\big\}$, which are functions of
$y$, $\lambda$, and $z$, form a basis of solutions to the differential
equation\footnote{The equation \eqref{eq:KP2PF} is the Picard--Fuchs
  equation associated to the Landau--Ginzburg mirror to $Y$.  The fact
  that the quantum cohomology of $Y$ can be determined from this
  Picard--Fuchs equation has been proved many times from many
  different points of view: see \emph{e.g.}
  \citelist{\cite{Givental:elliptic}\cite{Chiang--Klemm--Yau--Zaslow}\cite{Lian--Liu--Yau:I}\cite{Elezi}}.}
\begin{align}
  \label{eq:KP2PF}
  D_y^3 f &= y (\lambda - 3 D_y)(\lambda - 3 D_y-z)(\lambda - 3 D_y-2z)
   f, & D_y = z y \textstyle {\partial \over \partial y}.
\end{align}
Writing
\[
I_\cX(x,z) = 
I^0_\cX \, \phi_0 + 
I^1_\cX \, \phi_1 + 
I^2_\cX \, \phi_2,
\]
the components $\big\{I^j_\cX: j=0,1,2\big\}$, which are functions of
$x$, $\lambda$, and $z$, form a basis of solutions to the differential
equation
\begin{align}
  \label{eq:C3Z3PF}
  D_x^3 f &= {-27}x^{-3}(\lambda + D_x)(\lambda + D_x - z)(\lambda +
  D_x -2 z)f, & D_x = z x \textstyle {\partial \over \partial x}.
\end{align}

Recall that the functions $I_Y^j$ are defined in a region where $|y|$
is small.  The change of variables \eqref{eq:KP2gluing} turns
\eqref{eq:KP2PF} into \eqref{eq:C3Z3PF}.  This implies that if we
analytically continue the functions $I_Y^j$ to a region where $|y|$ is
large (and hence $|x|$ is small), and then write the analytic
continuations $\widetilde{I}^j_Y$ in terms of the co-ordinate $x$,
then $\{\widetilde{I}^j_Y(x,z): j=0,1,2\big\}$ will satisfy
\eqref{eq:C3Z3PF}.  We have a basis of solutions to \eqref{eq:C3Z3PF},
given by the components $I_\cX^k(x,z)$ of $I_\cX$, and so
\begin{equation}
  \label{eq:KP2match}
  \begin{pmatrix}
    \widetilde{I}^0_Y(x,z) \\
    \widetilde{I}^1_Y(x,z) \\
    \widetilde{I}^2_Y(x,z)
  \end{pmatrix}
  = M(\lambda,z)
  \begin{pmatrix}
    I_\cX^0(x,z) \\
    I_\cX^1(x,z) \\
    I_\cX^2(x,z) 
  \end{pmatrix}
\end{equation}
for some $3 \times 3$ matrix $M$ which is independent of $x$ and $y$
(and hence depends only on $\lambda$ and $z$).  The matrix
$M(\lambda,-z)$ defines the $\CC(\!(z^{-1})\!)$-linear symplectic
transformation $\U:\cHX \to \cHY$ which we seek.  It remains to
calculate the analytic continuations and to determine the matrix $M$.

\subsection*{Step 6: Analytic Continuation}

To compute the analytic continuation of $I_Y(y,z)$ we use the
Mellin--Barnes method.  Good references for this are
\citelist{\cite{Horja}\cite{CDGP}\cite{CCIT:crepant1}*{Appendix}}. 
First, take the expression \eqref{eq:KP2IGamma} for $I_Y$ and apply
the identity $\Gamma(x) \Gamma(1-x) = \pi/\sin(\pi x)$ until each
factor $\Gamma(a + b d)$ which occurs has $b>0$:
\begin{equation}
  \label{eq:IKP2before}
  I_Y(y,z) = - \Theta_Y \sum_{d \geq 0}
  {\Gamma\big(3d-\textstyle{\lambda - 3p \over z}\big) \over 
    \Gamma\big(1+{p \over z} + d\big)^3}
  (-1)^d \, y^{d+p/z}
\end{equation}
where
\[
\Theta_Y = \pi^{-1} z \,
\Gamma\big(1+\textstyle{p \over z}\big)^3 \,
\Gamma\big(1+\textstyle{\lambda - 3p \over z}\big) \,
\sin\big(\pi\big[\textstyle{\lambda - 3p \over z}\big]\big).
\]
Then, in view of \cite[Lemma 3.3]{Horja}, consider the contour
integral
\begin{equation}
  \label{eq:contourintegral}
  \int_C \Theta_Y 
  { 
    \Gamma\big(3s - {\lambda-3p \over z}\big) 
    \Gamma(s) \Gamma(1-s)
    \over
    \Gamma\(1 + {p \over z} + s\)^3
  } q^{s + p/z} 
\end{equation}
where the contour of integration $C$ is chosen as in
Figure~\ref{fig:contour}.
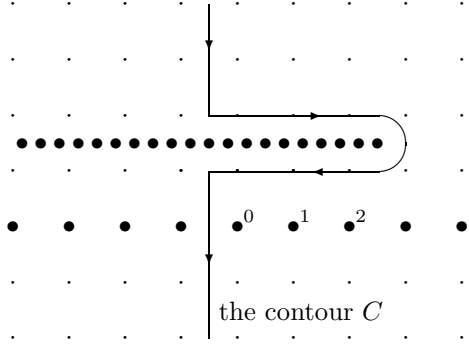
\begin{figure}[hbtp]
  \centering
  \unitlength 0.7pt
  \begin{picture}(260,200)(-130,-70)
    \multiput(-120,-60)(30,0){9}{\makebox(0,0){$\cdot$}}
    \multiput(-120,-30)(30,0){9}{\makebox(0,0){$\cdot$}}
    \multiput(-120,0)(30,0){9}{\makebox(0,0){$\bullet$}}
    \multiput(-120,30)(30,0){9}{\makebox(0,0){$\cdot$}}
    \multiput(-120,60)(30,0){9}{\makebox(0,0){$\cdot$}}
    \multiput(-120,90)(30,0){9}{\makebox(0,0){$\cdot$}}
    \multiput(-120,120)(30,0){9}{\makebox(0,0){$\cdot$}}
    \multiput(-30,0)(30,0){3}{\makebox(0,0){$\cdot$}}
    \multiput(-30,30)(30,0){3}{\makebox(0,0){$\cdot$}}
    \put(6,6){\makebox(0,0){$\scriptstyle 0$}}
    \put(36,6){\makebox(0,0){$\scriptstyle 1$}}
    \put(66,6){\makebox(0,0){$\scriptstyle 2$}}
    \multiput(75,45)(-10,0){20}{\makebox(0,0){$\bullet$}}
    \put(-15,30){\vector(0,-1){50}}
    \put(-15,-20){\line(0,-1){40}}
    \put(-15,20){\line(1,1){10}}
    \put(-15,60){\vector(1,0){60}}
    \put(45,60){\line(1,0){30}}
    \put(75,30){\vector(-1,0){35}}
    \put(40,30){\line(-1,0){55}}
    \put(-15,120){\vector(0,-1){25}}
    \put(-15,95){\line(0,-1){35}}
    \put(-15,70){\line(1,-1){10}}
    \put(75,45){\oval(30,30)[r]}
    \put(34,-45){\makebox(0,0){the contour $C$}}
  \end{picture}
  \caption{The contour of integration}
  \label{fig:contour}
\end{figure}
The integral \eqref{eq:contourintegral} is defined and
analytic throughout the region $|\arg(q)|<\pi$.  For $|q|<{1 \over
  27}$ we can close the contour to the right, and
\eqref{eq:contourintegral} is then equal to the sum of residues
\eqref{eq:IKP2before}.  For $|q|>{1 \over 27}$ we can close the
contour to the left, and then \eqref{eq:contourintegral} is equal to
the sum of residues at
\begin{align*}
  s = -1-n, & \quad n \geq 0, & \text{and} && s = \textstyle{\lambda - 3p \over 3z} -
  \textstyle{n \over 3}, & \quad n \geq 0.
\end{align*}
The residues at $s = -1-n$, $n \geq 0$, vanish in $H(Y)$ as they are
divisible by $p^3$.  Thus the analytic continuation $\widetilde{I}_Y$
of $I_Y$ is equal to the sum of the remaining residues:
\[
\Theta_Y \sum_{n \geq 0} {(-1)^n \over 3 . n!}
{ \pi \over \sin \big( \pi \big[\textstyle
  { \lambda - 3 p \over 3 z} - \textstyle{n \over 3} \big]
  \big)}
{1 \over \Gamma\big(1 + \textstyle {\lambda \over 3} - {n \over 3}
  \big)^3} \;
y^{\lambda/3z - n/3}.
\]
Writing this in terms of the co-ordinate $x$, we find that the
analytic continuation $\widetilde{I}_Y(x,{-z})$ is equal to
\begin{equation}
  \label{eq:IKP2ac}
  {-z}\, x^{\lambda /z} 
  \sum_{n \geq 0} {(-x)^n \over 3.n!}
  {\sin \big(\pi \big[ {\lambda - 3p \over z}\big] \big) \over
    \sin \big(\pi \big[ {\lambda - 3p \over 3z} + {n \over 3}\big]
    \big)}
  {\Gamma\big(1 - {p \over z}\big)^3 \over 
    \Gamma\big(1 - {\lambda \over 3 z} - {n \over 3}\big)^3} \,
  \Gamma\big(1-\textstyle{\lambda - 3p \over z}\big).
\end{equation}

\subsection*{Step 7: Compute the Symplectic Transformation} Our final
step is to compute the linear symplectic transformation $\U:\cHX \to
\cHY$ represented by the matrix $M(\lambda,-z)$.  We have
$\U(I_\cX(x,-z)) = \widetilde{I}_Y(x,-z)$, and
\begin{equation}
  \label{eq:IC3Z3firstfew}
  I_\cX(x,-z) = {-z}\, x^{\lambda/z} \Bigg(\fun_0 - {x \over z} \fun_{1/3} +
  {x^2 \over 2z^2}\fun_{2/3} + O(x^3)\Bigg).
\end{equation}
As the transformation $\U$ does not depend on $x$, we can compute it
by equating powers of $x$ in \eqref{eq:IKP2ac} and
\eqref{eq:IC3Z3firstfew}:
\begin{align*}
  & \U(\fun_0) = 
  {1 \over 3}
  {\sin \big(\pi \big[ {\lambda - 3p \over z}\big] \big) \over
    \sin \big(\pi \big[ {\lambda - 3p \over 3z}\big]
    \big)}
  {\Gamma\big(1 - {p \over z}\big)^3 \over 
    \Gamma\big(1 - {\lambda \over 3 z}\big)^3} \,
  \Gamma\big(1-\textstyle{\lambda - 3p \over z}\big)  \\
  & \U(\fun_{1/3}) = 
  {z \over 3}
  {\sin \big(\pi \big[ {\lambda - 3p \over z}\big] \big) \over
    \sin \big(\pi \big[ {\lambda - 3p \over 3z} + {1 \over 3}\big]
    \big)}
  {\Gamma\big(1 - {p \over z}\big)^3 \over 
    \Gamma\big(1 - {\lambda \over 3 z} - {1 \over 3}\big)^3} \,
  \Gamma\big(1-\textstyle{\lambda - 3p \over z}\big) \\
  & \U(\fun_{2/3}) = 
  {z^2 \over 3}
  {\sin \big(\pi \big[ {\lambda - 3p \over z}\big] \big) \over
    \sin \big(\pi \big[ {\lambda - 3p \over 3z} + {2 \over 3}\big]
    \big)}
  {\Gamma\big(1 - {p \over z}\big)^3 \over 
    \Gamma\big(1 - {\lambda \over 3 z} - {2 \over 3}\big)^3} \,
  \Gamma\big(1-\textstyle{\lambda - 3p \over z}\big). 
\end{align*}
The matrix $M$ of $\U$ does not have a simple form, but in the
non-equivariant limit it becomes
\begin{equation}
  \label{eq:UC3Z3nonequivariant}
  \begin{pmatrix}
    1 & 0 & 0 \\
    0 & -\frac{2 \pi }{\sqrt{3} \Gamma \left(\frac{2}{3}\right)^3} & -\frac{2 \pi  z}{\sqrt{3} \Gamma \left(\frac{1}{3}\right)^3} \\
    -\frac{\pi ^2}{3 z^2} & -\frac{2 \pi ^2}{3 z \Gamma \left(\frac{2}{3}\right)^3} & \frac{2 \pi ^2}{3 \Gamma \left(\frac{1}{3}\right)^3}
  \end{pmatrix}.
\end{equation}

From this point of view it is not obvious \emph{a priori} that $\U$ is
a symplectomorphism, or that it satisfies conditions (a) and (c) in
Conjecture~\ref{CRC} --- this is one advantage of the more
sophisticated approach taken in
\citelist{\cite{CCIT:crepant1}\cite{Iritani:inprogress}} --- but now that we
have an explicit expression for $\U$ it is easy to check these things.

\begin{thm}[The Crepant Resolution Conjecture for
  \protect{$\big[\CC^3/\ZZ_3\big]$}] \label{thm:C3Z3CRC}
  Conjecture~\ref{CRC} holds for $\cX = \big[\CC^3/\ZZ_3\big]$, $Y =
  K_{\PP^2}$.  
\end{thm}

\begin{proof}
  It remains only to check that, after analytic continuation, $\U$
  maps $\cLX$ to $\cLY$.  But $\U$ was constructed so as to map
  $I_\cX$ to the analytic continuation of $I_Y$, and $\cLX$
  (respectively $\cLY$) is uniquely determined by the fact that $x
  \mapsto I_\cX(x,-z)$ is a family of elements of $\cLX$ (respectively
  that $y \mapsto I_Y(y,-z)$ is a family of elements of $\cLY$).  Thus
  $\U$ maps $\cLX$ to the analytic continuation of $\cLY$.
\end{proof}

\begin{cor}[The Cohomological Crepant Resolution Conjecture for
  \protect{$\big[\CC^3/\ZZ_3\big]$}] \label{cor:C3Z3CCRC}
  The algebra obtained from the $T$-equivariant small quantum
  cohomology algebra of $Y = K_{\PP^2}$ by analytic continuation in
  the parameter $q_1$ followed by the specialization $q_1 = 1$ is
  isomorphic to the $T$-equivariant Chen--Ruan orbifold cohomology of
  $\cX = \big[\CC^3/\ZZ_3\big]$.
\end{cor}

\begin{proof}
  The quantity $c^1$ defined in \eqref{eq:defofci} is zero.  Now apply
  Corollary~\ref{cor:CCRC}.
\end{proof}

\begin{rem*}
  The symplectic transformation \eqref{eq:UC3Z3nonequivariant} with
  $z=1$ looks similar to the symplectic transformation computed by
  Aganagic--Bouchard--Klemm in \cite{ABK}, but it is not the same.  It
  would be interesting to understand the source of the discrepancy.
\end{rem*}

\section{Example II: $\cX = K_{\PP(1,1,3)}$}
\label{sec:KP113}

In this example we take $\cX:=K_{\PP(1,1,3)}$ to be the total space of
the canonical bundle of the weighted projective space $\PP(1,1,3)$ and
$Y$ to be the crepant resolution of the coarse moduli space of $\cX$.
We use essentially the same methods as before.  The slight changes in
method are needed to cope with the fact that, unlike all the other
examples considered in this paper, the non-compact toric variety $Y$
is \emph{not} presented as the total space of a vector bundle.

\subsection*{Toric Geometry}

Consider the action of $(\Cstar)^2$ on $\CC^5$ such that $(s,t) \in
(\Cstar)^2$ acts as 
\begin{equation}
  \label{eq:KP113action}
  (x,y,z,u,v)
  \longmapsto
  (tx,ty,sz,st^{-3}u,s^{-2}tv)
\end{equation}
The secondary fan is shown in
Figure~\ref{fig:KP113resolutionsecondaryfan}; the roman numerals there
label the different chambers. There is an exact sequence:
\[
\begin{CD}
  0 @>>> \ZZ^2 @>{
    \begin{pmatrix}
      0 & 1 \\ 0 & 1 \\ 1 & 0 \\ 1 & -3 \\ -2 & 1
    \end{pmatrix}}>>  
  \ZZ^5 @>{
    \begin{pmatrix} \textstyle
      1 & -1 & 0 & 0 & 0\\
      0 & 3 & -1 & 1 & 0 \\
      1 & 1 &  1 & 1 & 1
    \end{pmatrix}}>> \ZZ^3 @>>> 0.
\end{CD},
\]
and so each chamber in the secondary fan corresponds to a toric
orbifold with fan equal to some triangulation of the rays $e_1+e_3$,
${-e_1} + 3 e_2 + e_3$, ${-e_2} + e_3$, $e_2+e_3$, $e_3$.

\begin{figure}[htp]
  \centering
     \begin{picture}(90,110)(-45,-55)
      \multiput(-40,20)(20,0){4}{\makebox(0,0){$\cdot$}}
      \multiput(-40,0)(20,0){4}{\makebox(0,0){$\cdot$}}
      \multiput(-40,-20)(20,0){4}{\makebox(0,0){$\cdot$}}
      \multiput(-40,-40)(20,0){4}{\makebox(0,0){$\cdot$}}
      \multiput(-40,-60)(20,0){4}{\makebox(0,0){$\cdot$}}
      \put(0,1){\vector(0,1){20}}
      \put(0,-1){\vector(0,1){20}}
      \put(0,0){\vector(1,0){20}}
      \put(0,0){\vector(1,-3){20}}
      \put(0,0){\vector(-2,1){40}}
      \put(0,0){\makebox(0,0){$\bullet$}}
      \put(22,6){\makebox(0,0){$\scriptstyle 3$}}
      \put(2,26){\makebox(0,0){$\scriptstyle 1,2$}}
      \put(24,-58){\makebox(0,0){$\scriptstyle 4$}}
      \put(-42,26){\makebox(0,0){$\scriptstyle 5$}}
      \put(25,25){\makebox(0,0){I}}
      \put(25,-20){\makebox(0,0){II}}
      \put(-25,-20){\makebox(0,0){III}}
      \put(-20,25){\makebox(0,0){IV}}
    \end{picture}
  \caption{The secondary fan for $Y$}
  \label{fig:KP113resolutionsecondaryfan}
\end{figure}
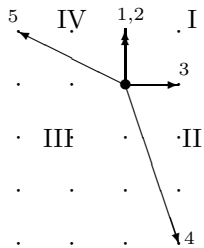

\noindent The fans are cones over the following pictures in the plane
$z=1$:
\begin{figure}[htbp]
  \centering
  \subfloat[I]{
    \begin{picture}(60,100)(-30,-30)
      \multiput(-20,-20)(20,0){3}{\makebox(0,0){$\cdot$}}
      \multiput(-20,0)(20,0){3}{\makebox(0,0){$\cdot$}}
      \multiput(-20,20)(20,0){3}{\makebox(0,0){$\cdot$}}
      \multiput(-20,40)(20,0){3}{\makebox(0,0){$\cdot$}}
      \multiput(-20,60)(20,0){3}{\makebox(0,0){$\cdot$}}
      \put(20,0){\makebox(0,0){$\bullet$}}
      \put(25,5){\makebox(0,0){$\scriptstyle 1$}}
      \put(-20,60){\makebox(0,0){$\bullet$}}
      \put(-25,65){\makebox(0,0){$\scriptstyle 2$}}
      \put(0,-20){\makebox(0,0){$\bullet$}}
      \put(-5,-24){\makebox(0,0){$\scriptstyle 3$}}
      \put(0,20){\makebox(0,0){$\bullet$}}
      \put(-5,23){\makebox(0,0){$\scriptstyle 4$}}
      \put(0,0){\makebox(0,0){$\bullet$}}
      \put(4,-4){\makebox(0,0){$\scriptstyle 5$}}
      \put(20,0){\line(-1,1){20}}
      \put(20,0){\line(-2,3){40}}
      \put(0,20){\line(-1,2){20}}
      \put(-20,60){\line(1,-3){20}}
      \put(-20,60){\line(1,-4){20}}
      \put(0,0){\line(1,0){20}}
      \put(0,0){\line(0,1){20}}
      \put(0,0){\line(0,-1){20}}
      \put(0,-20){\line(1,1){20}}
    \end{picture}
    \label{fig:KF2fan}
  }
  \qquad 
  \subfloat[II]{
    \begin{picture}(60,100)(-30,-30)
      \multiput(-20,-20)(20,0){3}{\makebox(0,0){$\cdot$}}
      \multiput(-20,0)(20,0){3}{\makebox(0,0){$\cdot$}}
      \multiput(-20,20)(20,0){3}{\makebox(0,0){$\cdot$}}
      \multiput(-20,40)(20,0){3}{\makebox(0,0){$\cdot$}}
      \multiput(-20,60)(20,0){3}{\makebox(0,0){$\cdot$}}
      \put(20,0){\makebox(0,0){$\bullet$}}
      \put(25,5){\makebox(0,0){$\scriptstyle 1$}}
      \put(-20,60){\makebox(0,0){$\bullet$}}
      \put(-25,65){\makebox(0,0){$\scriptstyle 2$}}
      \put(0,-20){\makebox(0,0){$\bullet$}}
      \put(-5,-24){\makebox(0,0){$\scriptstyle 3$}}
      \put(0,0){\makebox(0,0){$\bullet$}}
      \put(4,-4){\makebox(0,0){$\scriptstyle 5$}}
      \put(20,0){\line(-2,3){40}}
      \put(-20,60){\line(1,-3){20}}
      \put(-20,60){\line(1,-4){20}}
      \put(0,0){\line(1,0){20}}
      \put(0,0){\line(0,-1){20}}
      \put(0,-20){\line(1,1){20}}
    \end{picture}
    \label{fig:KP113fan}
  }
  \qquad 
  \subfloat[III]{
    \begin{picture}(60,100)(-30,-30)
      \multiput(-20,-20)(20,0){3}{\makebox(0,0){$\cdot$}}
      \multiput(-20,0)(20,0){3}{\makebox(0,0){$\cdot$}}
      \multiput(-20,20)(20,0){3}{\makebox(0,0){$\cdot$}}
      \multiput(-20,40)(20,0){3}{\makebox(0,0){$\cdot$}}
      \multiput(-20,60)(20,0){3}{\makebox(0,0){$\cdot$}}
      \put(20,0){\makebox(0,0){$\bullet$}}
      \put(25,5){\makebox(0,0){$\scriptstyle 1$}}
      \put(-20,60){\makebox(0,0){$\bullet$}}
      \put(-25,65){\makebox(0,0){$\scriptstyle 2$}}
      \put(0,-20){\makebox(0,0){$\bullet$}}
      \put(-5,-24){\makebox(0,0){$\scriptstyle 3$}}
      \put(20,0){\line(-2,3){40}}
      \put(-20,60){\line(1,-4){20}}
      \put(0,-20){\line(1,1){20}}
    \end{picture}
     \label{fig:C3Z5fan}
  }
  \qquad 
  \subfloat[IV]{
    \begin{picture}(60,100)(-30,-30)
      \multiput(-20,-20)(20,0){3}{\makebox(0,0){$\cdot$}}
      \multiput(-20,0)(20,0){3}{\makebox(0,0){$\cdot$}}
      \multiput(-20,20)(20,0){3}{\makebox(0,0){$\cdot$}}
      \multiput(-20,40)(20,0){3}{\makebox(0,0){$\cdot$}}
      \multiput(-20,60)(20,0){3}{\makebox(0,0){$\cdot$}}
      \put(20,0){\makebox(0,0){$\bullet$}}
      \put(25,5){\makebox(0,0){$\scriptstyle 1$}}
      \put(-20,60){\makebox(0,0){$\bullet$}}
      \put(-25,65){\makebox(0,0){$\scriptstyle 2$}}
      \put(0,-20){\makebox(0,0){$\bullet$}}
      \put(-5,-24){\makebox(0,0){$\scriptstyle 3$}}
      \put(0,20){\makebox(0,0){$\bullet$}}
      \put(-5,23){\makebox(0,0){$\scriptstyle 4$}}
      \put(20,0){\line(-1,1){20}}
      \put(20,0){\line(-2,3){40}}
      \put(0,20){\line(-1,2){20}}
      \put(-20,60){\line(1,-4){20}}
      \put(0,-20){\line(1,1){20}}
      \put(0,-20){\line(0,1){40}}
    \end{picture}
    \label{fig:KP122fan}
  }
  \caption{The fans corresponding to chambers~I--IV are
    the cones over these pictures in the plane $z=1$}
  \label{fig:KP113cones}
\end{figure}
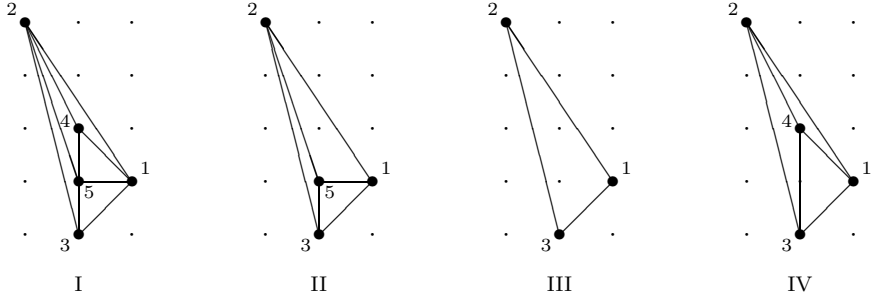

The toric orbifold corresponding to a chamber $C$ in the secondary fan
is the GIT quotient $\CC^5 \GIT{\xi} (\Cstar)^2$, $\xi \in C$.  This
is produced by deleting an appropriate union of co-ordinate subspaces
from $\CC^5$ and then taking the quotient by the action
\eqref{eq:KP113action}. When $C$ is chamber~I, the corresponding toric
orbifold is the resolution $Y$; chamber~II gives rise to the canonical
bundle $\cX = K_{\PP(1,1,3)}$; chamber~III gives the orbifold
$\big[\CC^3/\ZZ_5\big]$ where $\ZZ_5$ acts on $\CC^3$ with weights
$(1,1,3)$; and chamber~IV gives the canonical
bundle $K_{\PP(1,2,2)}$.
\begin{table}[htbp]
  \centering
  \begin{tabular}{@{}ccc@{}} \toprule
    \multicolumn{1}{c}{chamber} & 
    \multicolumn{1}{c}{locus to delete} & 
    \multicolumn{1}{c}{quotient} \\ \midrule
    I & $ \{z = u = 0\} \cup \{x = y = z =0\} \cup \{x=y=v=0\} $ & $Y$ \\[0.2em]
    II & $\{u = 0\} \cup \{x = y = z = 0\}$ & $\cX = K_{\PP(1,1,3)}$ \\[0.2em]
    III & $\{ u = 0\} \cup \{ v = 0\}$ & $\big[\CC^3/\ZZ_5\big]$ \\[0.2em]
    IV & $\{v = 0\} \cup \{x = y = z=0\}$ & $K_{\PP(1,2,2)}$ \\ \bottomrule
  \end{tabular}
  \caption{The different GIT quotients given by the secondary fan for
    $Y$}
  \label{tab:KP113quotients}
\end{table}

\noindent In this section we study the crepant resolution 
\begin{equation}
  \label{eq:KP113diagram}
  \xymatrix{
    Y \ar[rd] && K_{\PP(1,1,3)} \ar[ld] \\
    & X &}
\end{equation}
induced by moving from chamber~I to chamber~II.  In the next section
we consider the crepant partial resolution 
\[
  \xymatrix{
    K_{\PP(1,1,3)} \ar[rd]&& \big[\CC^3/\ZZ_5\big] \ar[ld] \\
    &\CC^3/\ZZ_5&
  }
\]
obtained by moving from chamber~II to chamber~III.  We will not
discuss chamber~IV at all.

\subsection*{The $T$-Action}

The action of $T = \Cstar$ on $\CC^5$ such that $\alpha \in
T$ maps
\[
(x,y,z,u,v)
\longmapsto
(x,y,z,u,\alpha v)
\]
descends to give actions of $T$ on $\cX$, $X$, and $Y$.  The induced
action on $\cX$ is the canonical $\Cstar$-action on the line bundle
$K_{\PP(1,1,3)} \to \PP(1,1,3)$; it covers the trivial action on
$\PP(1,1,3)$.  The crepant resolution \eqref{eq:KP113diagram} is
$T$-equivariant.

\subsection*{Bases for Everything}

We have
\begin{align*}
  & r := \rank H^2(Y;\CC) = 2, &
  & s := \rank H^2(\cX;\CC) = 1.
\end{align*}
Let $p_1, p_2 \in H(Y)$ denote the $T$-equivariant Poincar\'e-duals to
the divisors $\{z = 0\}$ and $\{x=0\}$ respectively, so that 
\[
H(Y) = \CC(\lambda)[p_1,p_2]/\big\langle p_2^2(\lambda+p_2 - 2p_1),
p_1(p_1-3p_2), p_1^2 p_2 \big\rangle.
\]
Set
\begin{align*}
  &\varphi_0 = 1, &
  &\varphi_1 = p_1, &
  &\varphi_2 = p_2, &
  &\varphi_3 = p_1 p_2, &
  &\varphi_4 = p_2^2.
\end{align*}

Write the inertia stack $\cIX$ of $\cX$ as the disjoint union $\cX_0
\coprod \cX_{1/3} \coprod \cX_{2/3}$, where $\cX_f$ is the component
of the inertia stack corresponding to the fixed locus of the element
$\big(1,e^{2 \pi \tti f}\big) \in (\Cstar)^2$.  We have $\cX_0 =
K_{\PP(1,1,3)}$ and $\cX_{1/3} = \cX_{2/3} = B\ZZ_3$.  Define $\fun_f
\in H(\cX)$ to be the class which restricts to the unit class on the
component $\cX_f$ and restricts to zero on the other components, and
let $p \in H(\cX)$ denote the first Chern class of the line bundle
$\cO(1) \to \PP(1,1,3)$, pulled back to $K_{\PP(1,1,3)}$ via the
natural projection and then regarded as an element of Chen--Ruan
cohomology via the inclusion $\cX =\cX_0 \to \cIX$.  Set
\begin{align*}
  &\phi_0 = \fun_0, &
  &\phi_1 = p, &
  &\phi_2 = p^2, &
  &\phi_3 = \fun_{1/3}, &
  &\phi_4 = \fun_{2/3},
\end{align*}
so that $r_1= {1 \over 3}$.

\subsection*{Step 1: A Family of Elements of $\cLY$}

Consider
\begin{multline}
  \label{eq:resolutionIGamma}
  I_Y(y_1,y_2,z) :=  z \sum_{k,l \geq 0}
  { \Gamma\big(1 + {p_2 \over z} \big)^2 
    \over \Gamma\big(1 + {p_2 \over z} + l \big)^2}
  { \Gamma\big(1 + {p_1 \over z} \big)
    \over \Gamma\big(1 + {p_1 \over z} + k \big)}
  { \Gamma\big(1 + {p_1 - 3 p_2 \over z} \big)
    \over \Gamma\big(1 + {p_1 - 3 p_2 \over z} + k - 3l \big)}
  \times \\
  { \Gamma\big(1 + {\lambda +p_2 - 2 p_1 \over z} \big)
    \over \Gamma\big(1 + {\lambda +p_2 - 2 p_1 \over z} +l - 2k \big)} \, 
  y_1^{k + p_1/z} y_2^{l + p_2/z}.
\end{multline}
This series converges, in a region where $|y_1|$ and $|y_2|$ are
sufficiently small, to a multi-valued analytic function of $(y_1,y_2)$
which takes values in $\cHY$.   We have:
\begin{multline*}
  I_Y(y_1,y_2,z) =  z \sum_{k,l \geq 0}
  {
    y_1^{k + p_1/z} y_2^{l + p_2/z} \over
    \prod_{m=1}^{m=l} (p_2 + m z)^2
    \prod_{m=1}^{m=k} (p_1 + m z)
  }
  { 
    \prod_{m \leq 0} (p_1 - 3 p_2 + m z) 
    \over
    \prod_{m \leq k-3l} (p_1 - 3 p_2 + m z) 
  } \\ \times
  { 
    \prod_{m \leq 0} (\lambda + p_2 - 2 p_1 + m z) 
    \over
    \prod_{m \leq l-2k} (\lambda + p_2 - 2 p_1 + m z) 
  }.
\end{multline*}
Note that all but finitely many terms in the infinite products here cancel.

\begin{proposition} \label{thm:KP113resolution}
 \begin{align*}
    I_Y(y_1,y_2,-z) \in \cLY && \text{for all $(y_1,y_2)$ in the
      domain of convergence of $I_Y$.}
  \end{align*}
\end{proposition}

\begin{proof}
  The argument which proves Theorem~0.1 in \cite{Givental:toric} also
  proves the claim here.  Theorem~0.1 as stated only applies to
  compact toric varieties, but the proof works for the non-compact
  toric variety $Y$ as well. The reader who would prefer not to check
  this should wait for the full generality of \cite{CCIT:stacks}.
\end{proof}

\subsection*{Step 2: $I_Y$ Determines $\cLY$}

We have:

\begin{cor} \label{cor:resolutionmirror}
  \[
  J_Y(q_1,q_2,z) = e^{\lambda g(y_1,y_2)/z} I_Y(y_1,y_2,z) 
  \]
  where:
  \begin{align*}
    & q_1 = y_1 \exp\big( 2g(y_1,y_2) - f(y_1,y_2)\big), 
    && q_2 = y_2 \exp\big( 3 f(y_1,y_2) - g(y_1,y_2) \big), \\
    & f(y_1,y_2) = 
    \sum_{\substack{0<l<\infty \\ 0 \leq k \leq  l/2}}
    {(-1)^{3l-k}(3l-k-1)! \over (l!)^2 k! (l-2k)!} y_1^k y_2^l,
    && g(y_1,y_2) = \sum_{\substack{0<k<\infty \\ 0 \leq l \leq k/3}}
    {(-1)^{2k-l}(2k-l-1)! \over (l!)^2 k! (k-3l)!} y_1^k y_2^l.
  \end{align*}
\end{cor}

\begin{proof}
  We argue exactly as in Corollary~\ref{cor:KP2mirror}.  Note that
  \[
  I_Y(y_1,y_2,z) = z + p_1 \big[\log y_1 - f(y_1,y_2) + 2 g(y_1,y_2)\big]
  + 
  p_2 \big[\log y_2 + 3 f(y_1,y_2) -g(y_1,y_2)\big] - \lambda g(y_1,y_2) + O(z^{-1}).
  \]
  It follows from Propositions~\ref{thm:stuff} and~\ref{thm:KP113resolution}
  that
  \[
  y \mapsto e^{{-\lambda} g(y_1,y_2)/z} I_Y(y_1,y_2,-z)
  \]
  is a family of elements of $\cLY$.  But
  \[
  e^{-\lambda g(y_1,y_2)/z} I_Y(y_1,y_2,-z) = -z + p_1 \log q_1 + p_2
  \log q_2 + O(z^{-1}),
  \]
  where $q_1$ and $q_2$ are defined above, and the unique family of
  elements of $\cLY$ of this form is $(q_1,q_2) \mapsto
  J_Y(q_1,q_2,-z)$.
\end{proof}

It follows, as before, that the series defining $J_Y(q_1,q_2,z)$
converges (to a multivalued analytic function) when $|q_1|$ and
$|q_2|$ are sufficiently small. Proposition~\ref{thm:stuff}b implies
that $\cLY$ is uniquely determined by the fact that $(y_1,y_2) \mapsto
I_Y(y_1,y_2,-z)$ is a family of elements of $\cLY$.

\subsection*{Aside: Computing Gromov--Witten Invariants of $Y$}

As in the previous example, one can invert the change of variables
$(y_1,y_2) \rightsquigarrow (q_1,q_2)$ and read off genus-zero
Gromov--Witten invariants of $Y$.  We will not do this.

\subsection*{Step 3: A Family of Elements of $\cLX$}

Let
\begin{multline}
  \label{eq:IKP113Gamma}
  I_\cX(x_1,x_2,z) := z \, 
  \sum_{\substack{d: d \geq 0,\\
      3d\in \ZZ}}
  \sum_{\substack{e: e \geq 0,\\
      3e\in \ZZ}}
  {x_1^{3d + 3p/z} x_2^{3e} \over (3e)! \, z^{3e} }
  { \prod_{\substack{b:b \leq 0 \\ \fr(b) = \fr(d-e)}} (p+b z)^2
    \over 
    \prod_{\substack{b:b \leq d-e \\ \fr(b) = \fr(d-e)}} (p+b z)^2}
  \times \\
  { \prod_{\substack{b:-5d-e<b\leq 0 \\ \fr(b) = \fr(-5d-e)}} (\lambda
    - 5p+b z)
    \over 
    \prod_{1 \leq m \leq 3 d} (3p+mz)}
  \fun_{\fr(e-d)}.
\end{multline} 
This converges, in some open set $\{(x_1,x_2) \in \Cstar \times \CC :
\text{$|x_1|$ and $|x_2|$ are sufficiently small}\}$, to a multivalued
analytic function which takes values in $\cHX$. 

\begin{proposition} \label{pro:KP113} 
  \begin{align*}
    I_\cX(x_1,x_2,-z) \in \cLX && \text{for all $(x_1,x_2)$ in the
      domain of convergence of $I_\cX$.}
  \end{align*}
\end{proposition}

\begin{proof}
  We first show that
  \begin{align}
    \label{eq:intermediate}
    I_\cX(x,0,-z) \in \cLX && \text{for all $(x,0)$ in the
      domain of convergence of $I_\cX$.}
  \end{align}
  For this we argue exactly as in Proposition~\ref{thm:KP2mirror},
  combining Theorem~\ref{thm:smalllinebundle} with
  \cite{CCLT}*{Theorem~1.7}.  Theorem~\ref{thm:smalllinebundle} here
  tells us how to modify the small $J$-function of $\PP(1,1,3)$ and
  Theorem~1.7 in \cite{CCLT} tells us how to compute the small
  $J$-function of $\PP(1,1,3)$.  Proposition~\ref{pro:KP113} then
  follows from \eqref{eq:intermediate} and Iritani's Reconstruction
  Theorem \cite{Iritani:gen}*{\emph{cf.} Example 4.15}.
\end{proof}

\subsection*{Step 4: $I_\cX$ Determines $\cLX$}

We have:

\begin{cor} \label{cor:KP113mirror}
    \[
    \JJ_\cX(\tau,z)\big|_{\tau = p \log q + r \fun_{1/3}}
    = e^{\lambda g(x_1,x_2)/z} I_\cX(x_1,x_2,z) 
    \]
    where
    \begin{align*} 
      & q = x_1^3 \exp\big(5 g(x_1,x_2)\big), &&
      g(x_1,x_2) = \sum_{\substack{0 \leq e \leq d; \\ 3d, 3e \in \ZZ; \\
          \fr(d) = \fr(e);\\(d,e) \ne (0,0)}}(-1)^{d+e} 
      {(5d+e-1)! \, x_1^{3d} x_2^{3e}
        \over 
        ((d-e)!)^2\,(3d)!\,(3e)!}, \\
      &  r = h(x_1,x_2) &
      & h(x_1,x_2) = \sum_{\substack{d,e\geq 0; \\ 3d, 3e \in \ZZ; \\
          \fr(e-d)=1/3}}
      { \Gamma\big({2 \over 3}\big)^3 x_1^{3d} x_2^{3e}
        \over 
        \Gamma(1+d-e)^2 \Gamma(1-5d-e)\,(3d)!\,(3e)!}.
    \end{align*}
\end{cor}

\begin{proof}
  Argue exactly as in Corollary~\ref{cor:KP2mirror}.
\end{proof}

This implies that the series $\JJ_\cX(\tau,z)\big|_{\tau = p \log q +
  r \fun_{1/3}}$ converges, in a region where $|q|$ and $|r|$ are
sufficiently small, to a multivalued analytic function of $q$ and $r$
which takes values in $\cHX$.  It also implies, via
Proposition~\ref{thm:stuff}b, that $\cLX$ is uniquely determined by
the fact that $(x_1,x_2) \mapsto I_\cX(x_1,x_2,-z)$ is a family of
elements of $\cLX$.

\subsection*{Aside: Computing Gromov--Witten Invariants of $\cX$}

We can use Corollary~\ref{cor:KP113mirror} to calculate genus-zero
Gromov--Witten invariants of $\cX$, by computing the first few terms
of the power series inverse to the ``mirror map'' $(x_1,x_2) \mapsto
(q,r)$.  This gives:
\begin{equation}
  \label{eq:KP113invariants}
  \begin{aligned}
    &\correlator{\fun_{1/3}}^\cX_{0,1,1/3} = -2, \\
    &\correlator{\fun_{1/3}}^\cX_{0,1,4/3} = {3757 \over 648}, \\
    &  \correlator{\fun_{1/3},\fun_{1/3}}^\cX_{0,1,2/3} = {-{13 \over
        18}}, \\
    &\correlator{\fun_{1/3},\fun_{1/3},\fun_{1/3}}^\cX_{0,1,0} = {1
      \over 3}, \\
    &\correlator{\fun_{1/3},\fun_{1/3},\fun_{1/3},\fun_{1/3}}^\cX_{0,1,1/3}
    = {-{2 \over 27}},\\
  \end{aligned}
\end{equation}
and so on.

\subsection*{Step 5: The $B$-model Moduli Space and the Picard--Fuchs
  System}

The $B$-model moduli space $\cM_B$ here is the toric orbifold
corresponding to the secondary fan for $Y$
(Figure~\ref{fig:KP113resolutionsecondaryfan}).  Each chamber of the
secondary fan gives a co-ordinate patch on $\cM_B$: the co-ordinates
$(y_1,y_2)$ coming from chamber~I are dual respectively to $p_1$ and
$p_2$, and the co-ordinates $(x_1,x_2)$ from chamber~II are dual
respectively to $p_1$ and $p_1 - 3 p_2$.  These co-ordinate patches
are related by
\begin{equation}
  \label{eq:KP113resolutiongluing}
  \begin{aligned}
    y_1 &= x_1 x_2 & \qquad \qquad \qquad x_1 &= y_1 y_2^{1/3} \\
    y_2 &= x_2^{-3} & \qquad \qquad \qquad x_2 &= y_2^{-1/3}.
  \end{aligned}
\end{equation}
We regard $I_Y(y_1,y_2,z)$ as a function on the co-ordinate patch
corresponding to chamber I and $I_\cX(x_1,x_2,z)$ as a function on the
co-ordinate patch corresponding to chamber II.  Let
\begin{align*}
  \textstyle D_{x_1} = z x_1 {\partial \over \partial x_1}, &&
  \textstyle  D_{x_2} = z x_2 {\partial \over \partial x_2}, &&
  \textstyle  D_{y_1} = z y_1 {\partial \over \partial y_1}, &&
  \textstyle  D_{y_2} = z y_2 {\partial \over \partial y_2}.
\end{align*}

The components of $I_Y(y_1,y_2,z)$, with respect to the basis
$\{\varphi_\alpha\}$, form a basis of solutions to the system of
differential equations:
\begin{equation}
  \label{eq:PFresolution}
  \begin{aligned}
    & D_{y_1}(D_{y_1} - 3D_{y_2}) f = 
    y_1 (\lambda + D_{y_2} - 2 D_{y_1})(\lambda + D_{y_2} - 2D_{y_1} - z) f \\
    & D_{y_2}^2(\lambda + D_{y_2} - 2D_{y_1}) f = 
    y_2 (D_{y_1} - 3D_{y_2})(D_{y_1} - 3D_{y_2} - z)(D_{y_1}-3D_{y_2}-2z)f.
  \end{aligned}
\end{equation}
and the components of $I_\cX(x_1,x_2,z)$, with respect to the basis
$\{\phi_\alpha\}$, form a basis of solutions to the system of
differential equations:
\begin{equation}
  \label{eq:PFKP113}
  \begin{aligned}
    & D_{x_2}(D_{x_2} - z)(D_{x_2} - 2z) f = \textstyle
    x_2^3 ({1 \over 3}D_{x_1} - {1 \over 3}D_{x_2})^2
    (\lambda - {5 \over 3}D_{x_1} - {1 \over 3}D_{x_2}) f \\
    & D_{x_1}D_{x_2}f = \textstyle
    x_1 x_2 (\lambda - {5 \over 3}D_{x_1} - {1 \over 3}D_{x_2}) 
    (\lambda - {5 \over 3}D_{x_1} - {1 \over 3}D_{x_2}-z) f \\
    & D_{x_1}(D_{x_1} - z)(D_{x_1} - 2z)  \textstyle
    ({1 \over 3}D_{x_1} - {1 \over 3}D_{x_2})^2 f =
    x_1^3 \prod_{k=0}^{k=4}  
    (\lambda - {5 \over 3}D_{x_1} - {1 \over 3}D_{x_2}-k z) f.
  \end{aligned}
\end{equation}
The change of variables \eqref{eq:KP113resolutiongluing} turns the
system of differential equations \eqref{eq:PFresolution} into the
system of differential equations \eqref{eq:PFKP113}.  It follows that
if $\widetilde{I}_Y$ is the analytic continuation of $I_Y$ to a
neighbourhood of $(x_1,x_2) = (0,0)$ then there exists a
$\CC(\!(z^{-1})\!)$-linear map $\U:\cHX \to \cHY$ such that
$\U(I_\cX(x_1,x_2,-z)) = \widetilde{I}_Y(x_1,x_2,-z)$.  As before, the
map $\U$ is the linear symplectomorphism that we seek.  To determine
it, we first calculate the analytic continuation $\widetilde{I}_Y$.

\subsection*{Step 6: Analytic Continuation}

To calculate $\widetilde{I}_Y$ we, for each $k \geq 0$, extract the
coefficient of $y_1^k$ from \eqref{eq:resolutionIGamma} and then analytically
continue it to a region where $|y_2|$ is large, using the
Mellin--Barnes method described in Section~\ref{sec:C3Z3}.  The result
is:
\begin{multline} \label{eq:KP113resolutionac}
    \widetilde{I}_Y(x_1,x_2,z) = 
    z \sum_{k,n \geq 0}
    {(-1)^{n+k} \over 3.n!}
    { \sin \big( \pi \big[ {p_1 -3 p_2 \over z} \big] \big)
      \over 
      \sin \big( \pi \big[ {p_1 -3 p_2 \over 3z}
      + {k-n \over 3} \big] \big)}
    { \Gamma \big(1 + {p_2 \over z} \big)^2
      \over 
      \Gamma \big(1 + {p_1 \over 3 z} + {k-n \over 3} \big)^2}
    \times \\
    { \Gamma \big(1 + {p_1 \over z} \big)
      \over 
      \Gamma \big(1 + {p_1 \over z} + k \big)}
    \Gamma \big(1 + \textstyle{p_1 - 3 p_2 \over z}\big)
    { \Gamma \big(1 + {\lambda - 2p_1 +p_2 \over z} \big)
      \over 
      \Gamma \big(1 + {3 \lambda - 5 p_1\over 3z} - {5 k+n \over 3} \big)}
    x_1^{k+{p_1 \over z}} x_2^n.
  \end{multline}

\subsection*{Step 7: Compute the Symplectic Transformation} 

Since $\U(I_\cX(x_1,x_2,-z)) = \widetilde{I}_Y(x_1,x_2,-z)$, we can
read off the transformation $\U$ by comparing coefficients of $x_1^a
x_2^b (\log x_1)^c$ in \eqref{eq:IKP113Gamma} and
\eqref{eq:KP113resolutionac}.  This gives:
\begin{align*}
  & \U(\fun_0) = 
  { \sin \big( \pi \big[ {p_1 -3 p_2 \over z} \big] \big)
      \over 
      3 \sin \big( \pi \big[ {p_1 -3 p_2 \over 3z}
      \big] \big)}
    { \Gamma \big(1 - {p_2 \over z} \big)^2
      \over 
      \Gamma \big(1 - {p_1 \over 3 z}\big)^2}
    { \Gamma \big(1 - {p_1 - 3 p_2 \over z}\big)
      \Gamma \big(1 - {\lambda - 2p_1 +p_2 \over z} \big)
      \over 
      \Gamma \big(1 - {3 \lambda - 5 p_1\over 3z} 
      \big)} \\
    & \U(p) = 
    { p_1 \over 3} \U(\fun_0)
    \\
    & \U(p^2) = 
    {p_1^2 \over 9} \U (\fun_0)
    \\
    & \U(\fun_{1/3}) = 
    {-{ z \sin \big( \pi \big[ {p_1 -3 p_2 \over z} \big] \big)
      \over 
      3 \sin \big( \pi \big[ {p_1 -3 p_2 \over 3z}
      +{1 \over 3} \big] \big)}}
    { \Gamma \big(1 - {p_2 \over z} \big)^2
      \over 
      \Gamma \big(1 - {p_1 \over 3 z} - {1 \over 3} \big)^2}
    { \Gamma \big(1 - {p_1 - 3 p_2 \over z}\big)
      \Gamma \big(1 - {\lambda - 2p_1 +p_2 \over z} \big)
      \over 
      \Gamma \big(1 - {3 \lambda - 5 p_1\over 3z} - {1 \over 3}
      \big)} \\
    & \U(\fun_{2/3}) = 
    { z^2 \sin \big( \pi \big[ {p_1 -3 p_2 \over z} \big] \big)
      \over 
      3 \sin \big( \pi \big[ {p_1 -3 p_2 \over 3z}
      + {2 \over 3} \big] \big)}
    { \Gamma \big(1 - {p_2 \over z} \big)^2
      \over 
      \Gamma \big(1 - {p_1 \over 3 z} - {2 \over 3} \big)^2}
    { \Gamma \big(1 - \textstyle{p_1 - 3 p_2 \over z}\big)
      \Gamma \big(1 - {\lambda - 2p_1 +p_2 \over z} \big)
      \over 
      \Gamma \big(1 - {3 \lambda - 5 p_1\over 3z} - {2 \over 3} \big)}
\end{align*}

The non-equivariant limit $\lim_{\lambda \to 0} \U$ has matrix:
\begin{equation}
  \label{eq:KP113nonequivariantU}
  \begin{pmatrix}
    1 & 0 & 0 & 0 & 0 \\
    0 & \frac{1}{3} & 0 & \frac{2 \pi }{3 \sqrt{3} \Gamma \left(\frac{2}{3}\right)^3} & -\frac{2 \pi  z}{3 \sqrt{3} \Gamma \left(\frac{1}{3}\right)^3} \\
    0 & 0 & 0 & -\frac{2 \pi }{\sqrt{3} \Gamma \left(\frac{2}{3}\right)^3} & \frac{2 \pi  z}{\sqrt{3} \Gamma \left(\frac{1}{3}\right)^3} \\
    \frac{\pi ^2}{9 z^2} & 0 & \frac{1}{3} & \frac{2 \pi ^2}{9 z \Gamma \left(\frac{2}{3}\right)^3} & \frac{2 \pi ^2}{9 \Gamma \left(\frac{1}{3}\right)^3} \\
    -\frac{\pi ^2}{3 z^2} & 0 & 0 & -\frac{2 \pi ^2}{3 z \Gamma \left(\frac{2}{3}\right)^3} & -\frac{2 \pi ^2}{3 \Gamma \left(\frac{1}{3}\right)^3}
  \end{pmatrix}.
\end{equation}

\begin{thm}[The Crepant Resolution Conjecture for
  \protect{$K_{\PP(1,1,3)}$}] \label{thm:KP113CRC} Conjecture~\ref{CRC}
  holds for $\cX = K_{\PP(1,1,3)}$ and $Y$ its crepant resolution.
\end{thm}

\begin{proof}
    Argue exactly as in the proof of Theorem~\ref{thm:C3Z3CRC}.
\end{proof}

\subsection*{Conclusions}
 
Having proved the Crepant Resolution Conjecture in this case, we can
now extract information about small quantum cohomology using
Corollary~\ref{cor:Ruan}.  When we do this, we find  that the quantum
corrections to Ruan's conjecture do not vanish:

\begin{cor}
  Let $\cX = K_{\PP(1,1,3)}$ and let $Y \to X$ be the crepant
  resolution of the coarse moduli space of $\cX$.  There is a power
  series 
  \[
  f(u) = {2 \pi \over \sqrt{3} \Gamma\big({1 \over 3}\big)^3} 
 \Bigg( {-2} u^{1/3} + {3757 \over 648} u^{4/3} + \cdots \Bigg)
 \]
 such that the algebra obtained from the small quantum cohomology
 algebra of $Y$ by analytic continuation in the parameter $q_2$
 followed by the substitution
  \[
  q_i =
  \begin{cases}
    e^{f(u)} u^{1/3} &  i=1 \\
    e^{-3f(u)} & i=2
  \end{cases}
  \]
  is isomorphic to the small quantum cohomology algebra of $\cX$, via
  an isomorphism which sends $p \in H(\cX)$ to ${1 \over 3} p_1 \in
  H(Y)$.
\end{cor}
\begin{proof}
  This is Corollary~\ref{cor:Ruan}.  The quantities $c_1$ and $c_2$
  defined in \eqref{eq:defofci} are zero, and the power series $f(u)$
  comes from equations \eqref{eq:defoffi} and
  \eqref{eq:KP113invariants}.
\end{proof}

\section{Example III: $\cX = \big[\CC^3/\ZZ_5\big]$,
  $\cY = K_{\PP(1,1,3)}$}
\label{sec:C3Z5}

We next consider an example of a crepant partial resolution.  Let
$\cX$ be the orbifold $\big[ \CC^3/\ZZ_5 \big]$ where $\ZZ_5$ acts on
$\CC^3$ with weights $(1,1,3)$.  The coarse moduli space $X$ of $\cX$
is the quotient singularity $\CC^3/\ZZ_5$, and a crepant partial
resolution $\cY$ of $X$ is the canonical bundle $K_{\PP(1,1,3)}$. We
make the obvious modifications to our general setup, replacing the
vector space $H(Y)$ with
\[
H(\cY) := H^\bullet_{\text{CR},T}(\cY;\CC) \otimes \CC(\lambda)
\]
and writing $Y$ for the coarse moduli space of $\cY$.  In this section
we omit all details, as the argument is completely parallel to that in
Sections~\ref{sec:C3Z3} and~\ref{sec:KP113}.

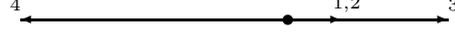
\begin{figure}[thbp]
  \centering
     \begin{picture}(180,10)(-110,-5)
      \multiput(-80,0)(20,0){7}{\makebox(0,0){$\cdot$}}
      \put(0,0){\makebox(0,0){$\bullet$}}
      \put(22,6){\makebox(0,0){$\scriptstyle 1,2$}}
      \put(62,6){\makebox(0,0){$\scriptstyle 3$}}
      \put(-102,6){\makebox(0,0){$\scriptstyle 4$}}
      \put(0,0){\vector(1,0){20}}
      \put(0,0){\vector(1,0){60}}
      \put(0,0){\vector(-1,0){100}}
    \end{picture}
    \caption{The secondary fan for $\cY = K_{\PP(1,1,3)}$}
    \label{fig:KP113secondaryfan}
\end{figure}

\noindent The secondary fan is shown in
Figure~\ref{fig:KP113secondaryfan}, the $B$-model moduli space $\cM_B$
is $\PP(1,5)$, and the $I$-functions are
\[
I_\cX(x_1,x_2,z) := 
z \, 
\sum_{k,l \geq 0} {x_1^k x_2^l  \over k! l! \, z^{k+l}}
\prod_{\substack{b : 0 \leq b < {k+2l  \over 5}, \\ 
    \fr{b} = \fr{k+2l\over 5}}} 
\big(\textstyle {\lambda \over 5} - b z)^2 \; 
\prod_{\substack{b : 0 \leq b < {3k+l  \over 5}, \\ 
    \fr{b} = \fr{3k+l  \over 5}}} 
\big(\textstyle {3\lambda \over 5} - b z) \;
\fun_{\fr({k+2l\over 5})}
\]
(\emph{c.f.} \cite{CCIT:computing}*{Theorem~4.6 and Proposition~6.1}) and
\begin{multline*}
  I_\cY(y_1,y_2,z) :=  z \, 
  \sum_{\substack{d: d \geq 0,\\
      3d\in \ZZ}}
  \sum_{\substack{e: e \geq 0,\\
      3e\in \ZZ}}
  {y_1^{3d + 3p/z} y_2^{3e} \over (3e)! \, z^{3e} }
  { \prod_{\substack{b:b \leq 0 \\ \fr(b) = \fr(d-e)}} (p+b z)^2
    \over 
    \prod_{\substack{b:b \leq d-e \\ \fr(b) = \fr(d-e)}} (p+b z)^2}
  \times \\
  { \prod_{\substack{b:-5d-e<b\leq 0 \\ \fr(b) = \fr(-5d-e)}} (\lambda
    - 5p+b z)
    \over 
    \prod_{1 \leq m \leq 3 d} (3p+mz)}
  \fun_{\fr(e-d)},
\end{multline*} 
(\emph{c.f.} Section~\ref{sec:KP113} above).  Use the bases 
\begin{align*}
  &\phi_0 = \fun_0, &
  &\phi_1 = \fun_{1/5}, &
  &\phi_2 = \fun_{2/5}, &
  &\phi_3 = \fun_{3/5}, &
  &\phi_4 = \fun_{4/5} \\
\intertext{for $H(\cX)$ and}
  &\varphi_0 = \fun_0, &
  &\varphi_1 = p, &
  &\varphi_2 = p^2, &
  &\varphi_3 = \fun_{1/3}, &
  &\varphi_4 = \fun_{2/3}
\end{align*}
for $H(\cY)$: for the notation see Sections~\ref{sec:C3Z3} and the
discussion above Theorem~\ref{thm:BZn}.  The Mellin--Barnes method
produces a linear symplectomorphism $\U:\cHX\to\cH_\cY$ such that
$\U(I_\cX(x_1,x_2,-z)) = \widetilde{I}_\cY(x_1,x_2,-z)$ where
$\widetilde{I}_\cY$ is the analytic continuation of $I_Y$.  In the
non-equivariant limit $\lambda \to 0$, the matrix of $\U$ is given by
\[
\begin{pmatrix}
  \scriptscriptstyle
  1 & \scriptscriptstyle 
  0 & \scriptscriptstyle 
  0 & \scriptscriptstyle 
  0 & \scriptscriptstyle 
  0 \\ \scriptscriptstyle 
  0 & \scriptscriptstyle 
  -\frac{\sqrt{2+\frac{2}{\sqrt{5}}} \pi }{\Gamma
    \left(\frac{2}{5}\right) \Gamma \left(\frac{4}{5}\right)^2} & \scriptscriptstyle
  \frac{25 \pi }{\sqrt{2 \left(5+\sqrt{5}\right)} \Gamma
    \left(-\frac{2}{5}\right)^2 \Gamma \left(\frac{4}{5}\right)} & \scriptscriptstyle
  -\frac{5 \pi  z}{\sqrt{2 \left(5+\sqrt{5}\right)} \Gamma
    \left(-\frac{4}{5}\right) \Gamma \left(\frac{2}{5}\right)^2} & \scriptscriptstyle
  \frac{\sqrt{\frac{5}{2} \left(5+\sqrt{5}\right)} \pi  z}{\Gamma
    \left(-\frac{2}{5}\right) \Gamma \left(\frac{1}{5}\right)^2} \\ \scriptscriptstyle 
  -\frac{\pi ^2}{z^2} & \scriptscriptstyle
  -\frac{5 \left(5+3 \sqrt{5}\right) \pi ^2}{z \Gamma
    \left(-\frac{1}{5}\right)^2 \Gamma \left(\frac{2}{5}\right)} & \scriptscriptstyle
  \frac{\left(-1+\frac{3}{\sqrt{5}}\right) \pi ^2}{z \Gamma
    \left(\frac{3}{5}\right)^2 \Gamma \left(\frac{4}{5}\right)} & \scriptscriptstyle
  \frac{\left(1-\frac{3}{\sqrt{5}}\right) \pi ^2}{\Gamma
    \left(\frac{1}{5}\right) \Gamma \left(\frac{2}{5}\right)^2} & \scriptscriptstyle
  \frac{\left(1+\frac{3}{\sqrt{5}}\right) \pi ^2}{\Gamma
    \left(\frac{1}{5}\right)^2 \Gamma \left(\frac{3}{5}\right)} \\ \scriptscriptstyle 
  \frac{\Gamma \left(\frac{2}{3}\right)^3}{5 z} & \scriptscriptstyle
  \frac{\sqrt{3} \csc \left(\frac{2 \pi }{15}\right) \Gamma
    \left(\frac{2}{3}\right)^3}{10 \Gamma \left(\frac{2}{5}\right)
    \Gamma \left(\frac{4}{5}\right)^2} & \scriptscriptstyle
  \frac{\sqrt{3} \csc \left(\frac{\pi }{15}\right) \Gamma
    \left(\frac{2}{3}\right)^3}{10 \Gamma \left(\frac{3}{5}\right)^2
    \Gamma \left(\frac{4}{5}\right)} & \scriptscriptstyle
  \frac{\sqrt{3} z \Gamma \left(\frac{2}{3}\right)^3 \sec
    \left(\frac{7 \pi }{30}\right)}{10 \Gamma \left(\frac{1}{5}\right)
    \Gamma \left(\frac{2}{5}\right)^2} & \scriptscriptstyle
  -\frac{\sqrt{3} z \Gamma \left(\frac{2}{3}\right)^3 \sec
    \left(\frac{\pi }{30}\right)}{10 \Gamma \left(\frac{1}{5}\right)^2
    \Gamma \left(\frac{3}{5}\right)} \\ \scriptscriptstyle 
  -\frac{\Gamma \left(\frac{1}{3}\right)^3}{5 z^2} & \scriptscriptstyle
  -\frac{\sqrt{3} \Gamma \left(\frac{1}{3}\right)^3 \sec
    \left(\frac{\pi }{30}\right)}{10 z \Gamma \left(\frac{2}{5}\right)
    \Gamma \left(\frac{4}{5}\right)^2} & \scriptscriptstyle
  \frac{\sqrt{3} \Gamma \left(\frac{1}{3}\right)^3 \sec \left(\frac{7
        \pi }{30}\right)}{10 z \Gamma \left(\frac{3}{5}\right)^2
    \Gamma \left(\frac{4}{5}\right)} & \scriptscriptstyle
  \frac{\sqrt{3} \csc \left(\frac{\pi }{15}\right) \Gamma
    \left(\frac{1}{3}\right)^3}{10 \Gamma \left(\frac{1}{5}\right)
    \Gamma \left(\frac{2}{5}\right)^2} & \scriptscriptstyle
  \frac{\sqrt{3} \csc \left(\frac{2 \pi }{15}\right) \Gamma
    \left(\frac{1}{3}\right)^3}{10 \Gamma \left(\frac{1}{5}\right)^2
    \Gamma \left(\frac{3}{5}\right)} 
\end{pmatrix}.
\]
Thus  Conjecture~\ref{CRC}
holds, exactly as stated, for the crepant \emph{partial} resolution
$\cY \to X$.

\begin{thm} \label{thm:C3Z5CRC}
  Conjecture~\ref{CRC} holds for $\cX = \big[\CC^3/\ZZ_5\big]$, $\cY =
  K_{\PP(1,1,3)}$. \qed  
\end{thm}

When we try to draw conclusions about small quantum cohomology,
however, a new phenomenon emerges.  For simplicity, let us discuss
this in the non-equivariant limit $\lambda \to 0$, indicating this by
a ($\ast$) following our equations.  In Section~\ref{sec:C3Z3}, when
we were considering $\cX=\big[\CC^3/\ZZ_3\big]$, we had
\[
\tag{$\ast$}
\U\big(\fun_{[\CC^3/\ZZ_3}\big) = \fun_{K_{\PP^2}} + O(z^{-2})
\]
and hence
\[
\tag{$\ast$}
\U\big(J_{[\CC^3/\ZZ_3]}(-z)\big) = -z \fun_{K_{\PP^2}} + O(z^{-1}).
\]
We can therefore identify $\U\big(J_{[\CC^3/\ZZ_3]}(-z)\big)$ with 
\[
\tag{$\ast$}
J_{K_{\PP^2}}(q,-z) = -z \fun_{K_{\PP^2}} + p \log q +
O(z^{-1})
\]
by setting $\log q = 0$, or in other words $q = 1$.  This is how the
specialization of quantum parameters in the Cohomological Crepant
Resolution Conjecture arises: see Corollary~\ref{cor:C3Z3CCRC} and
\cite{Coates--Ruan}.  In the case at hand, however, we have
\[
\tag{$\ast$}
\U\big(\fun_{[\CC^3/\ZZ_5]}\big) = \fun_{K_{\PP(1,1,3)}} + \textstyle
{1\over 5 }\Gamma({2 \over 3})^3  \fun_{1/3} + O(z^{-2})
\]
and thus
\[
\tag{$\ast$}
\U\big(J_{[\CC^3/\ZZ_5]}(-z)\big) = -z \fun_{K_{\PP(1,1,3)}} - \textstyle
{1\over 5 }\Gamma({2 \over 3})^3  \fun_{1/3}+ O(z^{-1}),
\]
which is not equal to the small $J$-function
$J_{K_{\PP(1,1,3)}}(q,-z)$ for any $q$ because the class $\fun_{1/3}$
\emph{comes from the twisted sector}.  We do have an equality
\[
\U\big(J_{[\CC^3/\ZZ_5]}(-z)\big) =
\JJ_{K_{\PP(1,1,3)}}(\tau,-z) \qquad \text{where} \qquad \tau = - \textstyle
{1\over 5 }\Gamma({2 \over 3})^3  \fun_{1/3},
\tag{$\ast$}
\]
but it does not let us conclude anything about small quantum
cohomology.  This is because there is no Divisor Equation for
Chen--Ruan classes from the twisted sector, so we cannot trade the
shift $\tau = 0 \rightsquigarrow \tau = c \fun_{1/3}$ for a
specialization $q \rightsquigarrow e^c$ (or indeed for any other
specialization of the quantum parameter).  

\subsection*{Conclusions}

In light of this, it seems likely that any generalization of the
Cohomological Crepant Resolution Conjecture (and hence also any
generalization of Ruan's Conjecture) to crepant partial resolutions
\emph{cannot be phrased in terms of small quantum cohomology alone}:
it must involve big quantum cohomology.  It seems also that any such
generalization will no longer involve only roots of unity.

\section{Example IV: A Toric Flop}
\label{sec:toricflop}

Finally, consider the action of $\Cstar$ on $\CC^5$ such that $s \in
\Cstar$ acts as
\begin{equation}
  \label{eq:flopaction}
  (x,y,z,u,v) 
  \longmapsto
  (sx,sy,sz,s^{-1}u,s^{-2}v)
\end{equation}
The secondary fan is:

\begin{figure}[thbp]
  \centering
     \begin{picture}(120,10)(-70,-5)
      \multiput(-60,0)(20,0){6}{\makebox(0,0){$\cdot$}}
      \put(0,0){\makebox(0,0){$\bullet$}}
      \put(22,6){\makebox(0,0){$\scriptstyle 1,2,3$}}
      \put(-22,6){\makebox(0,0){$\scriptstyle 4$}}
      \put(-42,6){\makebox(0,0){$\scriptstyle 5$}}
      \put(0,0){\vector(1,0){20}}
      \put(0,0){\vector(-1,0){20}}
      \put(0,0){\vector(-1,0){40}}
    \end{picture}
    \caption{The secondary fan for a toric flop}
    \label{fig:flopsecondaryfan}
\end{figure}
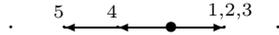

\noindent For $\xi$ in the right-hand chamber of the secondary fan,
the GIT quotient $Y := \CC^5 \GIT{\xi} \Cstar$ is the total space of
the vector bundle $\cO(-1) \oplus \cO(-2) \to \PP^2$.  For $\xi$ in
the left-hand chamber, the GIT quotient $\cX := \CC^5 \GIT{\xi}
\Cstar$ is the total space of $\cO(-1) \oplus \cO(-1) \oplus \cO(-1)
\to \PP(1,2)$.  The birational transformation $Y \dashrightarrow \cX$
induced by moving from the right-hand chamber to the left-hand chamber
is a flop \cite{Corti}.  

To treat this example, we need to make some changes to our general
setup (described in Section~\ref{sec:statement}), but the required
modifications are obvious and so we make them without comment.  As we
have not yet discussed a birational transformation of this type, we
once again give some details of the calculation: the reader will see
that our methods apply here too without significant change.

\subsection*{Bases and $I$-Functions}

We have
\begin{align*}
  & r := \rank H^2(Y;\CC) = 1, &
  & s := \rank H^2(\cX;\CC) = 1.
\end{align*}
The action of $T = \Cstar$ on $\CC^5$ such
that $\alpha \in T$ acts as
\[
(x,y,z,u,v)
\longmapsto
(\alpha x, \alpha y, \alpha z, u, v)
\]
induces actions of $T$ on $\cX$ and $Y$, and the flop $Y
\dashrightarrow \cX$ is $T$-equivariant.  Let $p$ be the canonical
$T$-equivariant lift of the first Chern class of the line bundle
$\cO(1) \to \PP^2$, so that
\[
H(Y) = \CC(\lambda)[p]/\langle p^3 \rangle.
\]
We use the basis
\begin{align*}
  & \varphi_0 = 1, &
  &  \varphi_1 = p, &
  & \varphi_2 = p^2
\end{align*}
for $H(Y)$.  The inertia stack of $\cX$ is the disjoint union $\cX_0
\coprod \cX_{1/2}$, where $\cX_0 = \cX$ and $\cX_{1/2} = B \ZZ_2$.
Let $\fun_f \in H(\cX)$ denote the class which restricts to the unit
class on the component $\cX_f$ and restricts to zero on the other
component, and let $\fp \in H(\cX)$ denote the canonical
$T$-equivariant lift of the first Chern class of the line bundle
$\cO(1) \to \PP(1,2)$, pulled back to $\cX$ via the natural projection
$\cX \to \PP(1,2)$ and then regarded as an element of Chen--Ruan
cohomology via the inclusion $\cX =\cX_0 \to \cIX$.  We use the basis
\begin{align*}
  &\phi_0 = \fun_0, &
  &\phi_1 = \fp, &
  &\phi_2 = \fun_{1/2} 
\end{align*}
for $H(\cX)$.

Let
\begin{align*}
&I_Y(y,z) = z \, \sum_{d \geq 0}
{
  \prod_{-2d < m \leq 0} ( 2 \lambda- 2 p + m z)
  \prod_{-d < m \leq 0} ( \lambda -  p + m z)
  \over \prod_{0 < m \leq d} ( p + m z)^3
} \, y^{d + p/z}, \\
\intertext{and let}
& I_\cX(x,z) =  z \, x^{-\lambda/z}
  \sum_{\substack{d: d \geq 0,\\
      2d\in \ZZ}}
  x^{d+\fp/z}
  {
    \prod_{\substack{b : -d < b \leq 0, \\ \fr(b) = \fr(-d)}}
    (\lambda -  \fp + m z)^3 
    \over
    \prod_{\substack{b : 0 < b \leq d, \\ \fr(b) = \fr(d)}}
    (\fp + b z)
    \prod_{1 \leq m \leq 2d} (2\fp+mz)
  }
  \fun_{\fr(-d)}
\end{align*}
Arguing exactly as before yields:

\begin{proposition} \label{thm:flopmirror} We have $I_Y(y,-z) \in
  \cLY$ for all $y$ such that $0<|y|<{1 \over 4}$, and $I_\cX(x,-z)
  \in \cLX$ for all $x$ such that $|x|<4$. \qed
\end{proposition}

\noindent Furthermore, as 
\begin{align*}
& x ^{-\lambda/z} I_\cX(x,-z) = -z + \fp \log x + O(z^{-1}) 
& \text{and} &
& I_Y(y,-z) = -z + p \log y + O(z^{-1}) 
\end{align*}
we conclude that:

\begin{cor}
    \begin{align*}
    & J_\cX(u,z) = x^{\lambda/z}I_\cX(u,z) &
    & \text{and} &
    & J_Y(q,z) = I_Y(q,z) .
  \end{align*}
  Note that the mirror maps here are trivial.\qed
\end{cor}

\noindent It follows that the Lagrangian submanifold-germs $\cLX$
and $\cLY$ are uniquely determined by
Proposition~\ref{thm:flopmirror}.

\subsection*{The $B$-model Moduli Space and Analytic Continuation}

The $B$-model moduli space $\cM_B$ here is $\PP^1$: it has a
co-ordinate patch (with co-ordinate $x$) corresponding to $\cX$ and a
co-ordinate patch (with co-ordinate $y$) corresponding to $Y$, related
by $y = x^{-1}$.  Regard $I_\cX(x,z)$ as a function on the co-ordinate
patch corresponding to $\cX$ and $I_Y(y,z)$ as a function on the
co-ordinate patch corresponding to $Y$, and denote by
$\widetilde{I}_Y(x,z)$ the analytic continuation of $I_Y$ to a
neighbourhood of $x = 0$.  As before, both $I_\cX$ and
$\widetilde{I}_Y$ have components which form a basis of solutions to
the Picard--Fuchs differential equation
\begin{align*}
  - x D^3 f &=  (\lambda + D)(2 \lambda + 2 D)(2 \lambda + 2 D- z)
  f, & D = z  x\textstyle {\partial \over \partial x}.
\end{align*}
It follows that there exists a $\CC(\!(z^{-1})\!)$-linear isomorphism
$\U:\cHX \to \cHY$ such that $\U(I_\cX(x,-z)) =
\widetilde{I}_Y(x,-z)$.  This is the linear symplectomorphism that we
seek.

The Mellin--Barnes method gives
\begin{multline*}
  \widetilde{I}_Y(x,z) = 
  z\, x^{-\lambda /z} 
  \sum_{k \geq 0} {x^{k+{1\over 2}} \over 2.(2k+1)!} 
  {\Gamma\big(1 +{p \over z}\big)^3 \over 
    \Gamma\big(1 + {\lambda \over z} - k - {1 \over 2}\big)^3}
  \Gamma\big({-k}- \textstyle{1 \over 2}\big)
  \Gamma\big(1 +  \textstyle{2 \lambda - 2p \over z}\big)
  \\
  \Gamma\big(1 +  \textstyle{\lambda - p \over z}\big)
  {\sin \big(\pi \big[ {\lambda - p \over z}\big] \big)
    \sin \big(\pi \big[ {2\lambda - 2p \over z}\big] \big) \over
    \pi \sin \big(\pi \big[ {\lambda - p \over z} - k  - {1 \over 2}\big]
    \big)} \\
  -z\, x^{-\lambda /z} 
  \sum_{k \geq 0} {x^k \over k!(2k)!} 
  {\Gamma\big(1 +{p \over z}\big)^3 \over 
    \Gamma\big(1 + {\lambda \over z} - k\big)^3}
  \Gamma\big(1 +  \textstyle{2 \lambda - 2p \over z}\big)
  \Gamma\big(1 +  \textstyle{\lambda - p \over z}\big)
  {\sin \big(\pi \big[ {2\lambda - 2p \over z}\big] \big) \over
    \pi }
  \\
  \Big( \textstyle H_{2k} + {H_k \over 2} -{ 3 \gamma \over 2} + {1 \over 2} \log
  y - {3 \over 2} \psi\big(1+{\lambda \over z} - k \big) - {\pi \over
    2} \cot\big(\pi\big[ {\lambda - p \over z}\big]\big) \Big),
\end{multline*}
where $\gamma$ is Euler's constant, $\psi(z)$ is the logarithmic
derivative of $\Gamma(z)$, and $H_k$ is the $k$th harmonic number.
Thus
\begin{align*}
  & \U\big(\fun_0\big) = 
  {-{\Gamma\big(1 -{p \over z}\big)^3 \over 
    \Gamma\big(1 - {\lambda \over z} \big)^3}}
  \Gamma\big(1 -  \textstyle{2 \lambda - 2p \over z}\big)
  \Gamma\big(1-  \textstyle{\lambda - p \over z}\big)
  {\sin \big(\pi \big[ {2\lambda - 2p \over z}\big] \big) \over
    \pi } 
  \Big( \textstyle { 3 \gamma \over 2} 
  + {3 \over 2} \psi\big(1-{\lambda \over z} \big) - {\pi \over
    2} \cot\big(\pi\big[ {\lambda - p \over z}\big]\big) \Big), \\
  & \U(\fp) = {- {z \over 2}\, {\Gamma\big(1 -{p \over z}\big)^3 \over 
    \Gamma\big(1 - {\lambda \over z} \big)^3}}
  \Gamma\big(1 -  \textstyle{2 \lambda - 2p \over z}\big)
  \Gamma\big(1 -  \textstyle{\lambda - p \over z}\big)
  {\sin \big(\pi \big[ {2\lambda - 2p \over z}\big] \big) \over
    \pi }, \\
  & \U\big(\fun_{1/2}\big) =
  {-{z^2 \over 4}}
  {\Gamma\big(1 -{p \over z}\big)^3 \over 
    \Gamma\big(1 - {\lambda \over z} - {1 \over 2}\big)^3}
  \Gamma\big({- \textstyle{1 \over 2}}\big)
  \Gamma\big(1 -  \textstyle{2 \lambda - 2p \over z}\big)
    \Gamma\big(1 -  \textstyle{\lambda - p \over z}\big)
  {\sin \big(\pi \big[ {\lambda - p \over z}\big] \big)
    \sin \big(\pi \big[ {2\lambda - 2p \over z}\big] \big) \over
    \pi \sin \big(\pi \big[ {\lambda - p \over z}  + {1 \over 2}\big]
    \big)}.
\end{align*}
Note that
\begin{equation}
  \label{eq:toricflopci}
  \begin{aligned}
    &\U\big(\fun_0\big) = 1 + O\big(z^{-2}\big),\\
    & \U(\fp) = (\lambda-p) + O\big(z^{-2}\big),\\
    & \U\big(\fun_{1/2}\big) = (\lambda - p)^2 + O\big(z^{-1}\big).
  \end{aligned}
\end{equation}
In the non-equivariant limit $\lambda \to 0$, our expressions for $\U$
simplify:
\begin{align*}
  & \U\big(\fun_0\big) \to 
  1 - {\pi^2 p^2 \over 3 z^2}, &
  & \U(\fp) \to {-p}, &
  & \U\big(\fun_{1/2}\big) \to p^2.
\end{align*}

\begin{thm}[A ``Flop Conjecture'' for $\cX$ and $Y$] \label{thm:toricflop}
  There is a choice of analytic continuations of $\cLX$ and $\cLY$
  such that, after analytic continuation, $\U\(\cLX\) = \cLY$.
  Furthermore $\U:\cHX \to \cHY$ is a degree-preserving
  $\CC(\!(z^{-1})\!)$-linear symplectic isomorphism which satisfies
  \begin{itemize}
  \item[(a)] $\U(\fun_\cX) = \fun_Y + O(z^{-1})$;
  \item[(c)] $\U\(\cHX^+ \)\oplus \cHY^- = \cHY$.
  \end{itemize}
\end{thm}

\begin{proof}
  Argue as in the proof of Theorem~\ref{thm:C3Z3CRC}.
\end{proof}

The transformation $\U$ does not satisfy any condition analogous to
property~(b) in Conjecture~\ref{CRC}, but we should not expect this.
Property~(b) arises from the fact that $\U$ intertwines certain
monodromies (let us call them \emph{the relevant monodromies}) of the
system of Picard--Fuchs equations coming from mirror symmetry: see
\cite{CCIT:crepant1}*{Proposition~4.7}.  In the case of toric crepant
resolutions the relevant monodromies generate $H^2(\cX)$, but for
general toric crepant birational transformations this is not the case.
The Mellin--Barnes method will always produce a transformation $\U$
which intertwines the relevant monodromies, but in the case at hand
this is vacuously true as the set of relevant monodromies is empty.
For a general flop
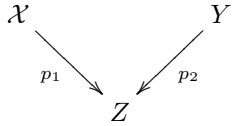
\[
\xymatrix{
  \cX \ar[rd]_{p_1} & & Y \ar[ld]^{p_2} \\
  & Z &}
\]
it is reasonable to expect that property~(b) should be replaced by the
assertion 
\begin{align*}
  \U\circ\big(p_1^\star\alpha \CR\big) = (p_2^\star\alpha \CR\big) \circ \U
  && \text{for all $\alpha \in H^2(Z;\CC)$;}
\end{align*}
this condition is also vacuous here.

\begin{cor}[A Ruan/Bryan--Graber-style Flop Conjecture]
  The $\CC(\lambda)$-linear map $\U_\infty:H(\cX) \to H(Y)$ given by
  \begin{align*}
    &\U_\infty\big(\fun_0\big) = 1, &
    & \U_\infty(\fp) = (\lambda-p), & 
    & \U_\infty\big(\fun_{1/2}\big) = (\lambda - p)^2,
  \end{align*}
  induces an algebra isomorphism between the small quantum cohomology
  of $\cX$ and the algebra obtained from the small quantum cohomology
  of $Y$ by analytic continuation in the quantum parameter $q$
  followed by the substitution $u = q^{-1}$.
\end{cor}

\begin{proof}
  Look at equation \eqref{eq:toricflopci}, and then apply the
  discussion in \cite{Coates--Ruan}*{\S9}.
\end{proof}

\end{document}